\documentclass[a4paper, 11pt]{amsart}
\usepackage{amsmath, amssymb,  thmtools, amsthm, array, enumerate, stmaryrd, enumitem, filecontents, mathtools}
\usepackage[all]{xy}
\usepackage[margin=1in]{geometry}

\usepackage{tikz-cd}

\usepackage{color}

\usepackage{hyperref}
\hypersetup{
    colorlinks=true,
    linkcolor=blue,
    urlcolor =blue,
    citecolor = blue}

\usepackage[noabbrev,capitalize, nameinlink]{cleveref}
\crefname{section}{section}{sections}
\Crefname{section}{section}{sections}
\crefname{subsection}{subsection}{subsections}
\Crefname{subsection}{subsection}{subsections}

\numberwithin{equation}{section}

\newtheorem{thmA}{Theorem}

\newtheorem{mythm}{Theorem}
\newtheorem{thm}{Theorem}[section]
\newtheorem{mylemma}[thm]{Lemma}
\newtheorem{myprop}[thm]{Proposition}
\newtheorem{mycor}[thm]{Corollary}
\newtheorem{mydef}{Definition}
\theoremstyle{remark}
\newtheorem{question}{Question}

\newtheorem*{remark}{Remark}
\newtheorem{example}{Example}

\def\Q2{\mathbb Q_2}

\def\Ga1{\Gamma_1}

\DeclareFontFamily{U}{wncy}{}
\DeclareFontShape{U}{wncy}{m}{n}{<->wncyr10}{}
 \DeclareSymbolFont{mcy}{U}{wncy}{m}{n}
 \DeclareMathSymbol{\Sh}{\mathord}{mcy}{"58}

\newcommand{\ZZ}{\mathbb Z}
\newcommand{\QQ}{\mathbb Q}
\newcommand{\FF}{\mathbb F}
\newcommand{\CC}{\mathbb C}

\newcommand{\ssll}{\mathop{\rm SL}_2(\ZZ)}
\newcommand{\gl}{\mathop{\rm GL}}
\newcommand{\lb}{\llbracket}
\newcommand{\rb}{\rrbracket}

\newcommand{\onto}{\twoheadrightarrow}
\newcommand{\into}{\hookrightarrow}
\mathchardef\myhyphen="2D
\DeclareMathOperator{\eend}{\rm End}
\DeclareMathOperator{\frob}{{\rm Frob}}
\DeclareMathOperator{\gal}{\rm Gal}

\newcommand{\old}{{\rm old}}

\newcommand{\new}{{\rm new}}

\newcommand{\ellcrit}{{\ell\myhyphen{\rm crit}}}
\newcommand{\ellnew}{{\ell\myhyphen\new}}
\newcommand{\ellold}{{\ell\myhyphen\old}}

\DeclareMathOperator{\Tr}{{{\rm Tr}_\ell}}

\newcommand{\pf}{{\rm pf}}
\newcommand{\full}{{\rm full}}
\newcommand{\eps}{\varepsilon}

\newcommand{\andand}{\quad \mbox{and} \quad}

\DeclareMathOperator{\SSS}{{{\mathcal D}_\ell}}
\DeclareMathOperator{\wt}{\mathcal S}
\DeclareMathOperator{\PS}{{\rm PS}}
\DeclareMathOperator{\hecke}{\underline{{\rm  Hecke}}}
\newcommand{\cond}{{\it Surj}}
\newcommand{\zerodiv}{{\it NZDiv}}
\newcommand{\monsky}{{\rm Monsky}}

\newcommand{\cmod}[1]{\ \mathrm{mod}\ #1}
\newcommand{\OO}{{\mathcal O}}
\newcommand{\uellnew}{{U_{\ell}\myhyphen\new}}
\newcommand{\tracenew}{{\Tr\myhyphen\new}}
\newcommand{\nl}{N\ell}
\newcommand{\ionl}{I_0(\nl)}
\newcommand{\yonl}{Y_0(\nl)_{\mathbb{F}_p}}
\newcommand{\xonl}{X_0(\nl)_{\mathbb{F}_p}}
\newcommand{\iord}{\ionl^{\text{ord}}}

\begin{document}

\parskip=7pt           
\parindent=0pt          
\relpenalty=10000        
\binoppenalty=10000      
\baselineskip 15pt
\title[Newforms mod $p$ in squarefree level]{Newforms mod $p$ in squarefree level \\
{\scriptsize with applications to Monsky's Hecke-stable filtration}}
\author[{S. V. Deo {\tiny and} A. Medvedovsky ({\tiny appendix by} A. Ghitza)}]{Shaunak V. Deo {\tiny and} Anna Medvedovsky \\ {\tiny with an appendix by} Alexandru Ghitza}
\address{University of Luxembourg, Mathematics Research Unit, Maison du nombre, 6, avenue de la Fonte, L-4364 Esch-sur-Alzette, Luxembourg}
\curraddr{School of Mathematics, Tata Institute of Fundamental Research, Homi Bhabha Road, Mumbai 400005, India}
\email{deoshaunak@gmail.com}
\address{Boston University, Department of Mathematics and Statistics, 111 Cummington Mall, Boston, MA 02215}
\email{medved@math.bu.edu}
\address{School of Mathematics and Statistics, University of Melbourne, Parkville VIC 3010, Australia}
\email{aghitza@alum.mit.edu}
\date{\today}
\thanks{Shaunak Deo was partially supported by the University of Luxembourg Internal Research Project AMFOR. Anna Medvedovsky was supported by an NSF postdoctoral fellowship (grant DMS-1703834).}

\setcounter{tocdepth}{1}
\begin{abstract}
We propose an algebraic definition of the space of $\ell$-new mod-$p$ modular forms for $\Gamma_0(N\ell)$ in the case that $\ell$ is prime to $N$, which naturally generalizes to a notion of newforms modulo $p$ in squarefree level. We use this notion of newforms to interpret the Hecke algebras on the graded pieces of the space of mod-$2$ level-$3$ modular forms described by Paul Monsky. Along the way, we describe a renormalized version of the Atkin-Lehner involution: no longer an involution, it is an automorphism of the algebra of modular forms, even in characteristic $p$.
\end{abstract}
\maketitle
\tableofcontents
%


\section{Overview}

This note is inspired by an explicit filtration on the space of modular forms modulo 2 of levels 3 and 5 described by Paul Monsky in \cite{monsky:level3fil, monsky:level5fil}, and our search for a conceptual description thereof. The goals of the present text are three-fold:
\begin{enumerate}
\item Develop an algebraic theory of spaces of $\ell$-new modular forms modulo $p$, consistent with the classical characteristic-zero definitions.
\item Introduce a modified Atkin-Lehner ``involution" that descends to an finite-order algebra automorphism of the space of modular forms modulo $p$. The appendix, written by Alex Ghitza, justifies this modification geometrically by viewing modular forms modulo $p$ as regular functions on the Igusa curve with poles only at supersingular points, and interpreting the Atkin-Lehner operator moduli-theoretically.
\item Construct a three-term Hecke-invariant filtration of the space of modular forms modulo $p$. On an old local component satisfying the level-raising condition at $\ell$, the Hecke algebras on the graded pieces of the filtration may be identified with two copies of the $\ell$-old Hecke algebra and one copy of the $\ell$-new Hecke algebra. We compare this filtration and its Hecke algebras to those found by Monsky in the case $\ell \equiv -1 \cmod{p}$.
\end{enumerate}
We now discuss each goal in detail. Throughout this section $N$ is an integer level, and $\ell$ is a prime dividing $N$ exactly once. The ring $B$ is a commutative $\ZZ[\frac{1}{\ell}]$-algebra.

\subsection{Spaces of $\ell$-new forms in characteristic $p$}
The theory of newforms in characteristic zero, developed by Atkin and Lehner \cite{atkinlehner}, traditionally casts new eigenforms as eigenforms that are \emph{not} old (i.e., do not come from lower level) and the space of newforms as a \emph{complement} (under the Petersson inner product) to the space of old forms. Alternatively, one can define what it means to be a new eigenform --- again, not old --- and then the newforms are those expressible as linear combinations of new eigenforms. Viewed from both perspectives, newforms are classically identified by what they are not rather than what they are: in a sense, a quotient space rather than a subspace.

This ``anti"-property of newforms creates problems as soon as we move into characteristic~$p$. On one hand, there is no Petersson inner product, so no obvious way to find a complement of the old forms. On the other hand, in fixed level, there are infinitely many forms modulo $p$, but only finitely many eigenforms, so we cannot rely on eigenforms alone to characterize the newforms. And even labeling mod-$p$ eigenforms as ``old" or ``new" is problematic, as newforms and oldforms in characteristic zero may admit congruences modulo $p$. 

We propose two different algebraic notions of newness in characteristic $p$, both based on properties of presence rather than absence. The first is based on the Atkin-Lehner result that an eigenform of level $N$ and weight $k$ that is new at a prime $\ell$ exactly dividing the level has its $U_\ell$-eigenvalue equal to $\pm \ell^\frac{k-2}{2}$ \cite[Theorem 3]{atkinlehner}. The second is inspired by an observation of Serre from \cite[\S3.1(d)]{SerreZetaPadic}: in the same setup, the $\ell$-new forms of level $N$ are exactly those forms $f$ that satisfy both $\Tr f = 0$ and $\Tr w_\ell f = 0$. Here $\Tr$ is the trace map from forms of level $N$ to forms of level $N/\ell$ (see \autoref{tracesec}), and $w_\ell$ is the Atkin-Lehner involution at $\ell$ (see \autoref{atkinlehnersec}).

More precisely, we define two submodules of $S_k(N, B)$, the module of cuspforms of weight~$k$ and level $N$ over $B$: let $S_k(N, B)^\uellnew$ be the kernel of the Hecke operator $U_\ell^2 - \ell^{k-2}$, and let $S_k(N, B)^\tracenew$ be the intersection of the kernels of $\Tr$ and $\Tr w_\ell$. Our first result is that these submodules coincide, and agree with the usual notion of $\ell$-newforms for characteristic-zero $B$:
\begin{thmA}[see \autoref{newformthm}, \autoref{newkerS}, \autoref{newtrtr}]\label{theorema}
  For any $\ZZ[\frac{1}{\ell}]$-domain $B$, we have $S_k(N, B)^{\uellnew} = S_k(N, B)^{\tracenew}$.
If $B \subset \CC$, then they both coincide with $S_k(N, B)^\ellnew$.
\end{thmA}

We give similar results for $S(N, B)$, the space of cuspforms of level $N$ and all weights over $B$, viewed as $q$-expansions (see \autoref{setupsec} for definitions), if $B$ is a domain. \autoref{theorema} allows us to define a robust notion of the module of $\ell$-new forms in characteristic $p$, and hence a notion of a module of newforms in characteristic $p$ for squarefree levels. 
Note that in characteristic $p$ the spaces of $\ell$-new and $\ell$-old forms need not be disjoint. The description of their intersection in \autoref{ellnewelloldsec} supports our definitions: this intersection matches the level-raising results of Ribet and Diamond \cite{ribet:levelraising, diamond:levelraising}, which describe conditions for mod-$p$ congruences between classical (i.e., characteristic-zero) $\ell$-new and $\ell$-old eigenforms. 

\subsection{Atkin-Lehner operators as algebra automorphisms on forms mod $p$}
It is well known that Atkin-Lehner operator $w_\ell$ (see \autoref{atkinlehnersec}) is an involution on $M_k(N, \ZZ[\frac{1}{\ell}])$, the space of modular forms of level $N$ and weight $k$ over $B = \ZZ[\frac{1}{\ell}]$, and descends to an involution on $M_k(N, \FF_p)$ as well. Less popular is the (easy) fact that $w_\ell$ is an \emph{algebra} involution of $M(N, \ZZ[\frac{1}{\ell}])$, the algebra of modular forms of level $N$ and all weights at once (here viewed as $q$-expansions; see \autoref{setupsec} for definitions). However, because of congruences between forms whose weights differ by an odd multiple of $p-1$, the Atkin-Lehner operator $w_\ell$ is not in general well-defined on $M(N, \FF_p)$, essentially because of the factor of $\ell^{\frac{k}{2}}$ that appears in its definition. In \autoref{atkinlehnersec} we discuss this difficulty in detail, and propose a renormalization $W_\ell$ of $w_\ell$ that does descend to an algebra automorphism of $M(N, \FF_p)$, with the property that $W_\ell^2$ acts on forms of weight $k$ by multiplication by~$\ell^k$.

In \autoref{alexghitza}, Alex Ghitza gives a geometric interpretation of the operator $W_\ell$ on $M(N,\mathbb{F}_p)$, constructing it from an automorphism of the Igusa curve covering the modular curve $\xonl$.

\subsection{Hecke-stable filtrations of generalized eigenspaces modulo $p$}
In the last part of the paper, we focus on using the space of $\ell$-new mod-$p$ cuspforms to get information about the structure of the mod-$p$ Hecke algebra of level $N$. We define a Hecke-stable filtration
of $K(N, \FF_p)$, the subspace of $S(N, \FF_p)$ annihilated by the $U_p$ operator (see \eqref{myfil}):
\begin{equation}
0 \subset K(N,\FF)_t^{\ellnew} \subset (\ker \Tr)_t \subset K(N, \FF)_t.
\end{equation}
Here the $t$ indicates that we've restricted to a generalized Hecke eigencomponent for the eigensystem carried by a pseudorepresentation $t$ landing in a finite extension $\FF$ of $\FF_p$ (see \autoref{heckealgsetup} for definitions).
If $t$ is $\ell$-old but satisfies the level-raising condition, then under certain regularity conditions on the Hecke algebra at level $N/\ell$, we show that the Hecke algebra on the graded pieces of this filtration are exactly $A(N, \FF)^\ellnew_t$, $A(N/\ell, \FF)_t$, $A(N/\ell, \FF)_t$, the shallow Hecke algebras acting faithfully on $K(N, \FF)^\ellnew_t$, $K(N/\ell, \FF)_t$, and $K(N/\ell, \FF)_t$, respectively. See \autoref{standardfil}.

Finally, we compare this filtration to the filtration given in the case $\ell \equiv -1 \cmod{p}$ by Paul Monsky in \cite{monsky:level3fil, monsky:level5fil} (see \eqref{monskyfil}):
\begin{equation}
0 \subset K(N/\ell, \FF)_t \subset (\ker \Tr)_t \subset K(N, \FF)_t.
\end{equation}
Here again $t$ marks an $\ell$-old component satisfying the level-raising condition. It is not difficult to see that the Hecke algebras on the first and third graded pieces are both $A(N/\ell, \FF)_t$. Under similar regularity conditions on $A(N/\ell, \FF)_t$, we show that the Hecke algebra on the middle graded piece is once again $A(N, \FF)^\ellnew_t$. See \autoref{monskyfiltruth}.

{\bf Wayfinding:} In \autoref{notationsec} we set the notation for the various spaces of modular forms that we consider. In \autoref{atkinlehnersec}, we discuss problems with the Atkin-Lehner operator in characteristic $p$ (when considering all weights at once) and introduce a modified version. In \autoref{tracesec} we discuss the trace-at-$\ell$ operator. In \autoref{oldformsec} we discuss $\ell$-old forms. In \autoref{newformsec} we discuss and propose a space of $\ell$-new forms over rings that are not subrings of $\CC$. Intersections between spaces of $\ell$-old and $\ell$-new forms, especially restricted to local components of the Hecke algebra (defined in \autoref{heckealgsetup}) are discussed in \autoref{ellnewelloldsec}. Finally in \autoref{monskysec}, we discuss two Hecke-stable filtrations and compare the Hecke algebras on the corresponding graded pieces.

{\bf Acknowledgements:} Anna Medvedovsky and Shaunak Deo thank Paul Monsky and Gabor Wiese for helpful conversations. Anna Medvedovsky and Alexandru Ghitza are grateful for the hospitality of the Max Planck Institute of Mathematics in Bonn during work on this project.

\section{Notation and setup}\label{notationsec}

\subsection{The space of modular forms with coefficients in $B$}\label{setupsec}
Fix $N \geq 1$. Let $M_k(N, \ZZ) \subset \ZZ \lb q \rb$ be the space of $q$-expansions of modular forms of level $\Gamma_0(N)$ and weight $k$ whose Fourier coefficients at infinity are integral. We define $M_k(N, B)$ for any commutative ring $B$ as $M_k(N, \ZZ) \otimes_\ZZ B$. By the $q$-expansion principle \cite[12.3.4]{diamondim}, the map $M_k(N, B) \to B\lb q \rb$ is injective, so that we may view $M_k(N, B)$ as a submodule of $B\lb q \rb$. Similarly, we let $S_k(N, \ZZ) \subset M_k(N, \ZZ)$ be the $q$-expansions at infinity of cuspidal modular forms of level $\Gamma_0(N)$ and weight $k$, and let $S_k(N, B) := S_k(N, \ZZ) \otimes_\ZZ B$, which we again view as a submodule of $B \lb q \rb$. Note that $S_k(N, B) \subset M_k(N, B)$.

Let $M(N, B) := \sum_{k = 0}^\infty M_k(N, B) \subset B\lb q \rb$, the algebra of all modular forms of level $N$ over $B$. If $B$ is a domain of characteristic zero, then this sum is direct and $M(N, B) = \bigoplus_{k = 0}^\infty M_k(N, B)$ (for $B \subset \CC$, this is \cite[Lemma 2.1.1]{miyake}; otherwise use the fact that $B$ is flat over $\ZZ$). On the other hand, if $B$ is a domain of characteristic $p$, then this sum is never direct: indeed, if $p \geq 5$, then a suitable multiple of the Eisenstein form $E_{p-1} \in M_{p-1}(1, B)$ has $q$-expansion~$1$, and therefore $M_k(N, B) \subset M_{k + p - 1} (N, B)$. This is essentially the only wrinkle: for $i \in 2\ZZ/(p-1)\ZZ$, set $$M(N, B)^i := \bigcup_{k \equiv i \cmod{p-1}} M_k(N, B);$$ then by \cite[Theorem 5.4]{goren}
\begin{equation}\label{gradingdecomp}
M(N, B) = \bigoplus_{i \in 2\ZZ/(p-1)\ZZ} M(N, B)^i,
\end{equation} making $M(N, B)$ into a $2\ZZ/(p-1)\ZZ$-graded algebra. If $p = 2, 3$ (and still $B$ has characteristic $p$) then multiples of both $E_4$ and $E_6$ have $q$-expansion $1$; certainly $M_k(N, B) \subset M_{k + 12}(N, B)$.

Similarly, let $S(N, B) := \sum_{k = 0}^\infty S_k(N, B) \subset B\lb q \rb$, the space of all cuspidal forms of level~$N$. This is a graded ideal of the graded algebra $M(N, B)$; let $S(N, B)^i$ be the $i^{\rm th}$ graded part, where $i \in \ZZ_{\geq 0}$ if $B$ has characteristic zero, $i \in 2\ZZ/(p-1) \ZZ$ if $p \geq 3$ and $i = 0$ if $p = 2$.

For any $f \in B \lb q \rb$ and $n \geq 0$, write $a_n(f)$ for the coefficient of $q^n$: that is, $f = \sum_{n \geq 0} a_n(f) q^n$. If $f \in M_k(N, B)$ or $M(N, B)$, then $a_n(f)$ is the $n^{\rm th}$ Fourier coefficient of $f$. For $m \geq 0$, write $U_m$ for the formal $B$-linear operator $B\lb q \rb \to B \lb q \rb$ given by $a_n(U_m f) = a_{mn}(f)$.

\subsection{Hecke operators on $M_k(N, B)$ and $M(N, B)$}
\label{hecke_operators}

The spaces $M_k(N, B)$ carry actions of the Hecke operators $T_m$, for all positive $m$ if $k \geq 2$ and for $m$ invertible in $B$ if $k = 0$. These Hecke operators satisfy $T_1 = 1$ and $T_m T_{m'} = T_{m m'}$ if $(m, m') = 1$, so that it suffices to define them for prime power $m$ only. If $f \in M_k(N, B)$ and $r$ is a prime not dividing $N$ (and again either $k \geq 2$ or $\frac{1}{r} \in B$), then the action of $T_{r^s}$ is determined by the definition of $T_r$ on $q$-expansions
\begin{equation}\label{Trqexp}
a_n(T_r f) = a_{r n}(f) + r^{k -1} a_{n/r} (f),
\end{equation} where we interpret $a_{n/r}(f)$ to be zero if $r \nmid n$, and the recurrence
\begin{equation}\label{Tprimepowerrec}
T_{r^s} = T_r T_{r^{s-1}} - r^{k-1} T_{r^{s - 2}}
\end{equation} for all $s \geq 0$. On the other hand, if $m$ divides $N$, then the action of $T_m$ on $f \in M_k(N, B)$ is given by $a_n(T_m f) = a_{mn}(f),$ so that $T_m$ coincides with the formal $U_m$ operator defined earlier. Finally, if the characteristic of $B$ is $c > 0$, and $m$ divides $c$, then the action of $T_m$ on $M_k(N, \ZZ)$ coincides with the action of $U_m$ so long as $k \geq 2$. We always work with and write $U_m$ instead of $T_m$ for $m$ dividing $N$ or the (positive) characteristic of $B$.

All of these classical Hecke operators commute with each other. Moreover, if $B$ is a domain, then all of them extend to the algebra of modular forms $M(N, B)$. Indeed, this is immediate if $B$ has characteristic zero (as $M(N, B)$ is the direct sum of the $M_k(N, B)$). If $B$ has characteristic $p$ and $r$ is a prime not dividing $Np$, then $T_r$ is well-defined on $M(N, B)$ from the $q$-expansion formula \eqref{Trqexp} because $M(N, B)$ is a direct sum of weight-modulo-$(p-1)$ spaces \eqref{gradingdecomp} and $r^{k-1}$ is well-defined in characteristic $p$ for $k$ modulo $p -1$. The action of $T_m$ on $M(N, B)$ for prime power $m$ relatively prime to $Np$ follows from the recurrence \eqref{Tprimepowerrec}. The action of $U_m$ for $m$ dividing $Np$ is independent of the weight and hence always well defined.

\subsubsection{Weight-separating operators}
\label{weight_separating}

We can streamline these arguments by introducing weight-separating operators. If $B$ is a domain and $m$ is invertible in $B$, we define the operator $\wt_m: M_k(N, B) \to M_k(N, B)$ by $\wt_m f : = m^k f$.\footnote{Caution: For $m$ prime to $N$ and the (positive) characteristic of $B$, many authors have historically worked with the weight-separating operator $S_m : = m^{k-2} \langle m \rangle$ on $M_k(\Gamma_1(N), B)$, where $\langle \cdot \rangle$ is the diamond operator. We  use a different normalization here so that $\wt_m$ extends to an algebra automorphism on $M(N, B)$. We will eventually work with $\wt_\ell$ for $\ell$ is a prime exactly dividing the level.} Note that $\wt_m$ extends to an \emph{algebra automorphism} of $M(N, B)$. If every $m$ prime to $N$ and the (positive) characteristic of $B$ is invertible in $B$ (for example, if $B$ is a $\QQ$-algebra or a finite extension of $\FF_p$), then the action of all the $T_m$ is generated by the action of the $T_r$ and $\wt_r$ for \emph{primes} $r$ not dividing $N$ or the (positive) characteristic of $B$.

\section{The Atkin-Lehner involution at $\ell$}\label{atkinlehnersec}

We now fix an additional prime $\ell$ not dividing $N$. From now on, we assume that $B$ is a $\ZZ[\frac{1}{\ell}]$-domain. Our eventual goal is to meaningfully compare the Hecke action on the algebras $M(N \ell, B)$ and $M(N, B)$. In this section, we discuss how to extend the Atkin-Lehner involution on $M_k(N \ell, B)$ to an \emph{algebra involution} on $M(N \ell, B)$.

\subsection{The Atkin-Lehner involution at $\ell$ in weight $k$}\label{atkinlehnerdef}
For $k \in 2 \ZZ_{\geq 0}$, we recall the definition and properties of the Atkin-Lehner involution on $M_k(N\ell, B)$ as in \cite{atkinlehner}.

Let $\mathcal H$ be the complex upper half plane. We extend the weight-$k$ right action of $\ssll$ on functions $f: \mathcal H \to \CC$ given by $\left. f \right|_k \gamma = j(\gamma, z)^{-k} f(\gamma z)$ to $\gamma \in \gl_2(\QQ)^+$ via
\begin{equation}\label{normal1}
(\left. f \right|_k \gamma)(z) = (\det \gamma)^{\frac{k}{2}} j(\gamma, z)^{-k} f(\gamma z).
\end{equation}
Here, for $\gamma = \begin{psmallmatrix} a & b \\ c & d \end{psmallmatrix} \in \gl_2(\QQ)^+$, we write $\gamma z$ for $\frac{az + b}{cz + d}$ (this is the usual conformal action of $\gl_2(\QQ)^+$ on $\mathcal H^+ = \mathcal H \cup \mathbb P^1(\QQ)$ leaving $\mathbb P^1(\QQ)$ invariant); and $j(\gamma, z) := cz + d$ is the usual automorphy factor. The normalization of $(\det \gamma)^{\frac{k}{2}}$ is chosen so that the scalars $\gl_2(\QQ)^+$ act trivially.

Let $\gamma_\ell \in \gl_2(\QQ)^+$ be any matrix of the form $\begin{psmallmatrix} \ell & a \\ N\ell & \ell b \end{psmallmatrix}$, where $a$ and $b$ are integers such that $\ell b -aN=1$, which can be found as we've assumed that $\ell \nmid N$. Let $w_\ell$ be the operator on functions $f : \mathcal H \to \CC$ sending $f$ to $\left.f \right|_k \gamma_\ell$. One can check that
\begin{enumerate}
\item \label{wlnormalizes} the matrix $\gamma_\ell$ normalizes $\Gamma_0(N\ell)$, so that $w_\ell$ maps $M_k(N \ell, \CC)$ to $M_k(N \ell, \CC)$;
\item any two choices of $\gamma_\ell$ differ by an element of $\Gamma_0(N \ell)$, so that the action of $w_\ell$ on $M_k(N \ell, \CC)$ is defined without ambiguity;
\item \label{wlinvolution} the matrix $\gamma_\ell^2 \begin{psmallmatrix} \ell & 0 \\ 0 & \ell \end{psmallmatrix}^{-1}$ is in $\Gamma_0(N \ell)$, and therefore $w_\ell$ is an involution, called the \emph{Atkin-Lehner involution} on $M_k(N \ell, \CC)$;
\item  for $N = 1$, the involution $w_\ell$ coincides with the \emph{Fricke involution} $f \mapsto \left. f \right|_k {\begin{psmallmatrix} 0 & -1 \\ \ell & 0 \end{psmallmatrix}}$;
\item \label{wellonold} if $f \in M_k(N, \CC) \subset M_k(N \ell, \CC)$, then $w_\ell f = \ell^{\frac{k}{2}} f(q^\ell)$;
\item \label{alint} $w_\ell$ is $\ZZ[\frac{1}{\ell}]$-integral, and is therefore defined as an involution on any $M_k(N \ell, B)$ so long as $B$ is a $\ZZ[\frac{1}{\ell}]$-algebra. (This statement relies on the geometric perspective of Atkin-Lehner induced on forms by a geometric involution of the modular curve $X_0(N\ell)$. See \cite[\S3.1(d)]{SerreZetaPadic} for $N=1$ and, for example, \cite[Theorem A.1]{conrad} for the general case.)
\end{enumerate}

\subsection{Atkin-Lehner as an algebra involution in characteristic zero}
If $B$ has characteristic zero, then it is clear from the definitions above and the direct sum property of $M(N\ell, B)$ that $w_\ell$ extends to an \emph{algebra} involution on $M(N\ell, B)$. However, if $B$ has characteristic $p$ and $\ell$ is not a square modulo $p$, then we incur a sign ambiguity, essentially because of the factor of $\ell^{\frac{k}{2}}$ coming from the determinant term in \eqref{normal1}.

In the next section, we discuss the extent to which the Atkin-Lehner involutions on $M_k(N\ell, B)$ patch together to an algebra involution on $M(N\ell, B)$ when $B$ has characteristic $p$.

\subsection{Atkin-Lehner as an algebra involution in characteristic $p$: difficulties}
In this section we work with $B = \FF_p$ and finite extensions. We also assume the theory of oldforms and newforms in characteristic zero \cite{atkinlehner}, which will be reviewed in \autoref{oldformsec} and \autoref{newformsec} below. From item \eqref{alint} above, we know that if $f$ and $f'$ are characteristic-zero modular forms of the same weight and level $N\ell$ that are congruent modulo $p$, then $w_\ell f$ and $w_\ell f'$ are congruent modulo $p$ as well. Indeed, this is what it means for $w_\ell$ to descend to an involution on $M_k(N\ell, B)$. However, if $f$ and $f'$ appear in weights that differ by an \emph{odd} multiple of $p-1$, then $w_\ell f$ will be congruent to $w_\ell f'$ up to a factor of $\left( \frac{\ell}{p} \right)$ only.
\subsubsection{Some bad examples}
There are examples in both newforms and oldforms.

\begin{enumerate}
\item {\bf Newform example:} Let $p$ be an odd prime. If $f \in M_{k}(N\ell, \ZZ_p)$ is a new eigenform, then $f$ is an eigenform for $w_\ell$ as well, so that $w_\ell f = \eps(f) f$ for some $\eps = \pm 1$. Moreover, $a_\ell(f) = - \eps(f) \ell^{\frac{{k} - 2}{2}}$ \cite{atkinlehner}. Suppose now that $f' \in M_{k'}(N\ell, \ZZ_p)$ is another new eigenform congruent to $f$, so that, in particular $a_\ell(f) \equiv a_\ell(g)$ modulo $p$. Now, $\eps(f) = - a_\ell(f) \ell^\frac{2-k}{2}$ and $\eps(f')=-a_\ell(f')\ell^\frac{2-k'}{2}$. So $\eps(f')$ will not be congruent to $\eps(f)$ modulo $p$ unless $\ell^\frac{k - k'}{2} \equiv 1 \pmod{p}$. In particular, if $p$ is odd and $k-k'$ is an odd multiple of $p-1$, then $\eps(f) \equiv \eps(f') \cmod{p}$ if and only if $\ell$ is a square modulo $p$.

For example, write $S_k(\ell, \QQ)^{\new, \pm}$ for the new subspace on which $w_\ell$ acts by $\pm 1$. For $\ell = 3$ the spaces $S_{12}(3, \QQ)^{+}$ and $S_{16}(3, \QQ)^{-}$ are one-dimensional, spanned by  \begin{align*}
f^+_{12} &= q + 78 q^2 - 243q^3 + 4036q^4 - 5370q^5 + O(q^6) \in S_{12}(3, \QQ)^{\new, +}\\
f^-_{16} &= q - 72 q^2 + 2187 q^3 - 27584 q^4 - 221490 q^5 + O(q^6) \in S_{16}(3, \QQ)^{\new, -}.
\end{align*}
\mbox{Then $f^+_{12}$ and $f^-_{16}$ are congruent $\cmod{5}$, but $w_3 f^+_{12} = f^+_{12}$ and $w_3 f^-_{16} = -f^-_{16}$ are not.}
\item {\bf Oldform example:} Let $f \in M_{k}(N, \ZZ_p)$ be any form, not necessarily eigen.
Then $w_\ell f = \ell^{\frac{k}{2}} f(q^\ell)$. Suppose $f' \in M_{k'}(N, \ZZ_p)$ is congruent to $f$. Then we similarly see that $w_\ell f \equiv w_\ell f' \cmod{p}$ if and only if either $\ell$ is a square modulo $p$ or $k - k'$ is a multiple of $2(p-1)$. Indeed, for any $p \geq 5$, compare $f = E_{p-1} \in M_{p-1}(1, \ZZ_p)$ and the constant form $f' = 1 \in M_0(1, \ZZ_p)$. Then $w_\ell E_{p-1} = \ell^\frac{p-1}{2} E_{p-1}(q^\ell)$ and $w_\ell(1) = 1$; these are congruent modulo $p$ exactly when $\left(\frac{\ell}{p}\right) = 1$.
\end{enumerate}

\subsubsection{And some good examples}\label{whengood}

As demonstrated above, it is not true in general that $w_\ell$ descends to an algebra involution of $M(N\ell, \FF)$. However it does work in certain cases:

\begin{enumerate}
\item \label{ellsquare} {\bf If $\ell$ is a square modulo $p$}, then there is no sign ambiguity, and $w_\ell$ is an algebra involution of $M(N\ell, \FF_p)$. This is easy to show by using multiplication by $E_{p-1}$ to move around different weights and using the fact that $w_\ell(E_{p-1}) = \left(\frac{\ell}{p} \right) E_{p-1}(q^\ell)$.
(Use $E_4$ and $E_6$ in place of $E_{p-1}$ if $p = 2$ or $3$.) In particular, $p =2$ never poses a problem.

\item \label{algM0} {\bf Restricting to $M(N \ell,\FF_p)^0$ and $p \geq 3$}, we can define $w_\ell$ as an algebra involution compatible with reduction of \emph{some} lift. Namely, $f \in M(N \ell, \FF_p)^0$ is the reduction of some $\tilde f \in M_k(N \ell, \ZZ_p)$ with $k$ divisible by $2(p-1)$; define $w_\ell f$ as the reduction of $w_\ell \tilde f$. Since any two such $\tilde f$s differ (multiplicatively) by a power of $E_{p-1}^2$, this construction is independent of the choice of $\tilde f$.\footnote{For $p \geq 5$, this construction is equivalent to the following geometric definition. By dividing $f \in M_{(p-1)k}(N\ell, \FF_p)$ by $E_{p-1}^{k}$, we can identify $M(N\ell, \FF_p)^0$ with the algebra of regular functions on the affine curve obtained by removing the supersingular points from $X_0(N\ell)_{\FF_p}$ (see Serre \cite[Corollaire 2]{SerreBBKmodp}). The geometric Atkin-Lehner involution on $X_0(N\ell)_{\FF_p}$ preserves the supersingular locus and hence induces an algebra involution on this geometrically defined $M(N\ell, \FF_p)^0$.}

\end{enumerate}

\begin{remark}
One can show that if $p \equiv 1$ modulo $4$ and $\ell$ is not a square modulo $p$, then there is no algebra involution on $M(N\ell, \FF_p)$ extending the involution on $M(N\ell, \FF_p)^0$ described in \eqref{algM0} above with the property that every $f \in M(N\ell, \FF_p)$ is sent to a reduction of $w_\ell \tilde f$ for \emph{some} lift $\tilde f \in M(N\ell, \ZZ_p)$ of $f$. Is the same true for $p \equiv 3$ modulo $4$?
\end{remark}

\subsection{Modified Atkin-Lehner as an algebra automorphism in characteristic $p$}\label{renormalizeal}

To fix this difficulty, we will renormalize $w_\ell$ to be compatible with algebra structures.

For any $m \in \ZZ$, possibly depending on $k$, the weight-$k$ right action of $\ssll$ on functions $f: \mathcal H \to \CC$ can be extended to $\gl_2(\QQ)^+$ via the formula, for $z \in \mathcal H$,
$$(\left. f \right|_{k, m} \gamma)(z) = (\det \gamma)^m j(\gamma, z)^{-k} f(\gamma z).$$
Scalar matrices $\begin{psmallmatrix} a & 0 \\ 0 & a \end{psmallmatrix}$ then act via multiplication by $a^{2m - k}$. The usual choice in the definition of the Atkin-Lehner operator is $m = \frac{k}{2}$ (scalars act trivially; see, for example, \cite[p.135]{atkinlehner}); another possibility that appears in the literature is $m = k - 1$ (used to define Hecke operators; see, for example, \cite[Exercise 1.2.11]{diamondshurman}). For our renormalized Atkin-Lehner operator, we adopt $m = k$, so that scalars act through their $k^{\rm th}$ power.

We define a new map
\begin{align*}
W_\ell: \textstyle M_k(N\ell, \ZZ[\frac{1}{\ell}]) &\to \textstyle M_k(N\ell, \ZZ[\frac{1}{\ell}])\\
f &\mapsto \left.f\right|_{k, k} \gamma_\ell.
\end{align*}
Here $\gamma_\ell$ is again a matrix of the form $\begin{psmallmatrix} \ell & a \\ N\ell & \ell b \end{psmallmatrix}$, where $a$ and $b$ are integers such that $\ell b -aN=1$, as in \autoref{atkinlehnerdef}.
Since $W_\ell = \ell^{\frac{k}{2}} w_\ell$, it is clear that this map is well-defined independent of the choice of $\gamma_\ell$. Moreover, $W_\ell$ satisfies the following properties.

\begin{myprop}\label{w properties}
\begin{enumerate}\leavevmode
\item $W_\ell$ extends to an automorphism of $M_k(N\ell, B)$ for any $\ZZ[\frac{1}{\ell}]$-algebra $B$, satisfying
\begin{enumerate}
\item\label{w^2} $W_\ell^2 = \wt_\ell$ (here $\wt_\ell$ is the weight-separating operator defined in \autoref{weight_separating}); 
\item\label{w action on oldforms} $W_\ell f = \wt_\ell f(q^\ell)$ and $W_\ell f(q^\ell) = f$ for $f \in M_k(N, B)$.
\end{enumerate}
\item $W_\ell$ extends to an algebra automorphism of $M(N\ell, B)$ for any characteristic-zero $\ZZ[\frac{1}{\ell}]$-domain $B$. This algebra automorphism preserves the ideal $S(N \ell, B)$.
\item $W_\ell$ descends to an algebra automorphism for any characteristic-$p$ domain $B$. This algebra automorphism restricts to the involution on $M(N\ell, \FF_p)^0$ defined in \autoref{whengood} \eqref{algM0}. For $p\geq 3$, the order of $W_\ell$ divides $p - 1$; for $p=2$, $W_\ell$ coincides with $w_\ell$ and hence has order $2$.
 \end{enumerate}
\end{myprop}

Only the last item requires justification. It relies on the following:

\begin{mylemma}
If $f, g \in M(N\ell, \ZZ_p)$ are congruent modulo $p$, then so are $W_\ell(f)$ and $W_\ell(g)$.
\end{mylemma}

\begin{proof}
It suffices to consider $f, g$ appearing in single weights, so let these be $k(f), k(g)$, respectively. Since $w_\ell$ already has this property for $k(f) = k(g)$, so does $W_\ell$. It therefore suffices prove the case $k(f) < k(g)$. By a theorem of Serre (see equation \eqref{gradingdecomp}) $k(g) - k(f) = n(p-1)$ for some $n \in \ZZ^+$. But then $E_{p-1}^nf$ and $g$ are congruent in the same weight, so $W_\ell(E_{p-1})^n W_\ell(f) \equiv W_\ell(g) \pmod{p}$. The observation that $W_\ell(E_{p-1}) = \ell^{p-1} E_{p-1}(q^\ell) \equiv 1$ modulo $p$ completes the proof.
\end{proof}

\autoref{alexghitza} shows that the renormalized Atkin-Lehner operator $W_\ell$ in characteristic $p$ is induced geometrically on modular forms by an automorphism of the Igusa curve.

\section{The trace from level $N\ell$ to level $N$}\label{tracesec}
For any characteristic-zero $\ZZ[\frac{1}{\ell}]$-domain $B$,
there is a $B$-linear trace operator $$\Tr: M_k(N\ell, B) \to M_k(N, B)$$ given, for $B = \CC$, by
\begin{equation}\label{deftrace}
\Tr(f) = \sum_{\gamma \in \Gamma_0(N\ell) \backslash \Gamma_0(N)} \left. f \right|_k \gamma,
\end{equation}
first studied for $N = 1$ by Serre in \cite[\S3.1.(c)]{SerreZetaPadic}.

One can show (\cite[Lemma $2.2$]{mcgrawono}, or \cite[\S3.1.(c)]{SerreZetaPadic} for $N=1$) that $\Tr f = f + \ell^{1 - \frac{k}{2}}\, U_\ell \,w_\ell\, f.$ Equivalently, \begin{equation}\label{trace!}
 \Tr f = f +  \wt_\ell^{-1}  \ell\, U_\ell W_\ell f.
\end{equation}
Equation \eqref{trace!} shows immediately that $\Tr$ extends to a $B$-linear operator $M(N\ell, B) \to M(N, B)$ for any $\ZZ[\frac{1}{\ell}]$-domain $B$. The following identities are adjusted from \cite[\S3.1.(c)]{SerreZetaPadic}. They are valid for any $\ZZ[\frac{1}{\ell}]$-domain $B$. In fixed weight $k$, they are valid for any $\ZZ[\frac{1}{\ell}]$-algebra $B$.
\begin{enumerate}\label{idtrace}
\item For $f \in S(N\ell, B)$, we have $\Tr f \in S(N, B)$.
\item For $f \in M(N\ell, B)$, we have $\Tr \wt_\ell  f = \wt_\ell f +  \ell\, U_\ell W_\ell f.$
\item For $f \in M(N\ell, B)$, we have $\Tr W_\ell f = W_\ell f + \ell \, U_\ell f$.
\item For $f \in M(N, B)$, we have $\Tr f = (\ell + 1) f$.
\item For $f \in M(N, B)$, we have $\ell\, T_\ell f = \ell\, U_\ell f + W_\ell f$.
\item For $f \in M(N, B)$, we have $\Tr W_\ell f = \ell\, T_\ell f$.
\end{enumerate}

The shape of these equations suggest that it might be more natural to renormalize $T_\ell$ and $U_\ell$ by scaling them by $\ell$, so that the Hecke operators are true ``trace" rather than a scaled trace and stay integral even in weight $0$. In fact, this renormalization would amount to using the $|_{k, k}$-action discussed in \autoref{renormalizeal} to define the Hecke operators, which we are already using to define $W_\ell$. But we will not do so here.

\section{The space of $\ell$-old forms}\label{oldformsec}

\subsection{Two copies of $M(N, B)$ in $M(N\ell, B)$}
There are two embeddings of $M(N, B)$ into $M(N\ell, B)$: the identity and $W_\ell$. First we study their intersection.

\begin{myprop}\label{oldintk}
For any $\ZZ[\frac{1}{\ell}]$-algebra $B$, if $f \in M_k(N, B) \cap W_\ell M_k(N, B)$, then $f$ is constant.
\end{myprop}

\begin{proof}

Let $g \in M_k(N, B)$ be such that $f = W_\ell(g)$. We use \autoref{w properties} \eqref{w^2} and \eqref{w action on oldforms} to see that $\ell^k g = W_\ell^2 (g) = W_\ell(f) = \ell^k f(q^\ell),$ so that
$$f = \ell^k g(q^\ell) = \ell^k f(q^{\ell^2}).$$
But this means that $f$ has to be a constant! Indeed, suppose $n > 0$ is the least integer such that $a_n(f) \neq 0$. Since the right-hand side is in $B \lb q^{\ell^2} \rb$, we must have $n = m \ell^2$ for some $m < n$. But the $q^{n}$-coefficient on the right-hand side is $\ell^k a_m(f)$, which must be zero as $n$ was the least index of a nonzero coefficient of $f$.
\end{proof}

Alternatively, we can deduce \autoref{oldintk} in characteristic zero from \cite[Theorem 1]{atkinlehner} and in characteristic $p$ from the following more recent theorem of Ono-Ramsey.

\begin{thm}[Ono-Ramsey, {\cite[Theorem 1.1]{onoramsey}}]\label{onoramsey}
Let $p$ be a prime, and $f$ a form in $M_k(N, \ZZ)$ with $\bar f  = \sum a_n q^n  \in M_k(N, \FF_p)$ its mod-$p$ image. Suppose that there exists an $m$ prime to $Np$ and a power series $g \in \FF_p\lb q \rb$ so that $\bar f = g(q^m)$.
Then $\bar f = a_0$.
\end{thm}

\begin{mycor}\label{oldint}
If $B$ is a $\ZZ[\frac{1}{\ell}]$-domain, then $M(N, B) \cap W_\ell M(N, B) = B \subset B \lb q \rb$.
\end{mycor}

\begin{proof}
Let $f, g \in M(N, B)$ be forms so that $f = W_\ell (g) \in B \lb q^\ell \rb$. In light of \autoref{oldintk}, it suffices to show that we may assume that both $f$ and $g$ appear in a fixed weight $k$.  As a $\ZZ[\frac{1}{\ell}]$-domain, $B$ is flat over either $\ZZ[\frac{1}{\ell}]$ or over $\FF_p$ for some $p$ prime to $\ell N$. In either case, from \autoref{setupsec}, we know that we can express both $f$ and $g$ as finite sums of forms $f = \sum f_i$ and $g = \sum g_i$ with $f_i, g_i \in M_{k_i}(N, B)$ for some weights $k_i$, with $M_{k_i}(N \ell, B)$ linearly independent inside $M(N\ell, B)$. Then $\sum f_i  = \sum W_\ell(g_i)$ forces $f_i = W_\ell(g_i)$ in a single weight $k_i$.
\end{proof}

\subsection{$\ell$-Old forms}
Following Atkin-Lehner \cite{atkinlehner} and others, define the $\ell$-old forms in $M_k(N\ell, \QQ)$ as the span of $M_k(N, \QQ)$ and $W_\ell M_k(N, \QQ)$:
\begin{equation}\label{oldq}
M_k(N \ell, \QQ)^\ellold: = M_k(N, \QQ) + W_\ell M_k(N, \QQ) \subset M_k(N\ell, \QQ).
\end{equation}
Note that both $M_k(N, \QQ)$ and $W_\ell M_k(N, \QQ)$ have bases in $\ZZ\lb q \rb$; therefore $M_k(N \ell, \QQ)^\ellold$ does as well. Let $M_k(N\ell, \ZZ)^\ellold$ be the forms in $M_k(N \ell, \QQ)^\ellold$ whose $q$-expansions are integral:
$$M_k(N\ell, \ZZ)^\ellold := M_k(N\ell, \QQ)^\ellold \cap \ZZ\lb q \rb,$$
and let $S_k(N \ell, \ZZ)^\ellold : = S_k(N \ell, \ZZ) \cap M_k(N \ell, \ZZ)^\ellold$ be the cuspidal submodule.
Finally, for any ring $B$, let $M_k(N \ell, B)^\ellold$ (respectively, $S_k(N\ell, B)^\ellold$) be the image of $M_k(N \ell, \ZZ)^\ellold \otimes_\ZZ B$ (respectively, $S_k(N \ell, \ZZ)^\ellold \otimes_\ZZ B$) inside $B \lb q \rb$. Our definitions are not self-contradictory: for $B = \QQ$ the definition of $M_k(N \ell, \QQ)^\ellold$ coincides with \eqref{oldq} because of its $\ZZ$-structure. For the same reason, $S_k(N \ell, B)^\ellold = S(N \ell, B) \cap M_k(N\ell, B)^\ellold$ for any $B$.

Note that $M_k(N \ell, B)^\ellold$ may a priori be bigger than $M_k(N, B) + W_\ell M_k(N, B)$. For example, if $E_k$ is the \emph{normalized} (i.e., with $a_1 = 1$) weight-$k$ level-one Eisenstein series and $B = \ZZ_p$, then $$E_{p-1}^\ellcrit := E_{p-1}(q) - E_{p-1}(q^\ell)$$
is in $M_k(\ell,B )^\ellold$ but not in $M_k(1, B) + W_\ell M_k(1, B)$, since $E_{p-1}$ is has $p$ in the denominator of its constant term.\footnote{In fact for $N = 1$ and $B = \ZZ_p$ or $\FF_p$ one can show that this is essentially the only such exception.} For our purposes, the following will suffice:

\begin{myprop}\label{oldcusp}
If $B$ is a $\ZZ[\frac{1}{\ell}]$-algebra, then $S_k(N\ell, B)^\ellold = S_k(N, B) \oplus W_\ell S_k(N, B)$.
\end{myprop}

\begin{proof}
Since we are in a single weight, it suffices to consider $B = \ZZ[\frac{1}{\ell}]$.

Certainly $S_k(N, B) \oplus W_\ell S_k(N, B)$ is contained in $S_k(N \ell, B)^\ellold$, and by \autoref{oldint} this sum is direct. For the other containment, any element of $S_k(N \ell,B)^\ellold$ looks like $f = M^{-1}\big(g + W_\ell (h)\big)$ for some $g, h \in S_k(N,B)$ and $M \in B$. Then the fact that $\frac{1}{M}\big(g + W_\ell (h)\big)$ is $B$-integral means that $g \equiv -W_\ell (h) \cmod{M B}$. But by \autoref{oldintk} applied to $B/MB$, we must have $$W_\ell (h) \equiv g \equiv a_0(g) \equiv 0 \cmod{M B},$$ so that both $\frac{1}{M} g$ and $\frac{1}{M}W_\ell (h)$ are in fact $B$-integral.
\end{proof}

Finally, let $M(N \ell, B)^\ellold := \sum_k M_k(N \ell, B)^\ellold \subset B \lb q \rb$, the space of $\ell$-old forms of any weight. Similarly, $S(N \ell, B)^\ellold := \sum_k S_k(N \ell, B)^\ellold$ is the submodule of $\ell$-old cuspforms.

\section{The space of $\ell$-new forms}\label{newformsec}
\subsection{$\ell$-New forms in characteristic zero}
\subsubsection{Analytic notion}
For $B = \CC$ one can follow Atkin-Lehner's characterization of newforms to define the space $S_k(N\ell, \CC)^\ellnew$ of cuspidal $\ell$-new forms of level $N\ell$ and weight $k$ as the orthogonal complement to the space of $\ell$-old forms under the Petersson inner product \cite[p. 145]{atkinlehner}. Alternatively, the space of $\ell$-new cuspforms is the $\CC$-span of the \emph{$\ell$-new eigenforms}: those eigenforms that are not in $S_k(N\ell, \CC)^\ellold$ \cite[Lemma 18]{atkinlehner}. This latter definition can be extended to Eisenstein forms as well, to obtain well-defined spaces $M_k(N\ell, \CC)^\ellnew$ and $S_k(N\ell, \CC)^\ellnew$, which we here identify with their $q$-expansions.

One can show that $M_k(N\ell, \CC)^\ellnew$ has a basis in $\ZZ \lb q \rb$ (since Galois conjugates of $\ell$-new eigenforms are $\ell$-new \cite[Corollary 12.4.5]{diamondim}, one can mimic the argument in  \cite[Corollary 6.5.6]{diamondshurman}; see also Brunault's answer to \href{https://mathoverflow.net/questions/109871/integral-basis-for-newforms}{MathOverflow question 109871}). Therefore, the definitions
$$M_k(N \ell, \ZZ)^\ellnew : = M_k(N \ell, \CC)^\ellnew \cap \ZZ \lb q \rb$$
and, for any characteristic-zero domain $B$,
$$M_k(N \ell, B)^\ellnew : = M_k(N \ell, \ZZ)^\ellnew \otimes_\ZZ B \subset B\lb q \rb$$
are compatible with the definition of $M_k(N \ell, \CC)^\ellnew$ above.
Finally, set $$M(N \ell, B)^\ellnew := \sum_k M_k(N \ell, B)^\ellnew \subset B\lb q \rb$$ as usual. In characteristic zero, of course, this sum is direct.

\subsubsection{First algebraic notion: $U_\ell$-eigenvalue}
Combining the Atkin-Lehner computations of $\ell$-new eigenvalues together with the Weil bound, one can obtain a purely algebraic characterization of the space of newforms. Define two operators $\SSS^\pm: M_k(N\ell, B) \to M_k(N\ell, B)$ via $\SSS^+ :=\ell U_\ell
 - \ell^{\frac{k}{2}}$ and $\SSS^- := \ell U_\ell + \ell^{\frac{k}{2}}$, and let $\SSS := \SSS^+ \SSS^-$. Since $\SSS =\ell^2 U_\ell^2 - \wt_\ell,$
 the operator $\SSS$ defines $B$-linear grading-preserving operator $M(N\ell, B) \to M(N\ell, B)$ and $S(N\ell, B) \to S(N\ell, B)$.

\begin{myprop}\label{newkerS} If $B$ is a domain of characteristic zero, then $$M(N\ell, B)^{\ellnew} = \ker \SSS.$$
\end{myprop}

We sketch a proof below, starting with a lemma that relies on the Ramanujan-Petersson Conjecture (``Weil bound"), implied by the Weil Conjectures, proved by Deligne.

\begin{mylemma}[Ramanujan-Petersson, Weil, Deligne] \label{weilbound}
If $g = \sum a_n q^n \in S_k(N, \CC)$ is a normalized Hecke eigenform, and $\ell$ is any prime, then $\left| a_\ell(g) \right| < (\ell + 1) \ell^{\frac{k-2}{2}}$.
\end{mylemma}

\begin{proof}
The negation of the inequality violates the
the Weil bound $\left| a_\ell(g) \right| \leq 2 \ell^{\frac{k - 1}{2}}$. Indeed, $(\ell + 1)\ell^{\frac{k - 2}{2}} \leq 2 \ell^{\frac{k - 1}{2}}$ is equivalent to $(\ell + 1)^2 \leq 4 \ell$, which cannot happen for $\ell > 1$.
\end{proof}

\begin{proof}[Proof of \autoref{newkerS}]
It suffices to prove that the kernel of $\left.\SSS\right|_{M_k(N\ell, B)}$ is $M_k(N\ell, B)^\ellnew$ in a single weight $k$. Moreover, since $B$ is flat over $\ZZ$ it suffices to prove the statement for $B = \ZZ$; and since $M_k(N\ell, \CC)^\ellnew$ has a basis over $\ZZ$, it suffices to take $B = \CC$.

The module $M_k(N\ell, \CC)$ is a direct sum of one-dimensional subspaces spanned by $\ell$-new eigenforms and two-dimensional subspaces $V_{\ell, g}$, each spanned by an eigenform $g \in M_k(N, \CC)$ and $W_\ell(g)$. Since each of these subspaces is $U_\ell$-invariant, it suffices to see that $\SSS$ annihilates all $\ell$-new eigenforms, and that  $\pm \ell^{\frac{k}{2} -1}$ is never a $U_\ell$-eigenvalue on any $V_{\ell, g}$. If $f \in M_k(N\ell, \CC)$ is an $\ell$-new eigenform, then it is cuspidal, and by \cite[Theorem 5]{atkinlehner} its $U_\ell$-eigenvalue is $\pm \ell^{\frac{k}{2} -1}$, so that $\SSS f = 0$. Now consider $V_{\ell, g}$ for some eigenform $g \in M_k(N, \CC)$. The characteristic polynomial of $U_\ell$ on $V_{\ell, g}$ is $P_{\ell, g}(X)=X^2-a_{\ell}(g)X+\ell^{k-1}$, where $a_{\ell}(g)$ is the $T_{\ell}$-eigenvalue of $g$; we aim to show that $\pm \ell^{\frac{k}{2}-1}$ is not a root of $P_{\ell,g}(X)$.
If $g$ is Eisenstein, then $a_\ell(g) = \chi(\ell) \ell^{k-1} + \chi(\ell)^{-1}$ for some Dirichlet character $\chi$ of modulus $M$ with $M^2 \mid N$ (see, for example, \cite[Theorem 4.5.2]{diamondshurman}), so that the absolute values of the roots of $P_{\ell,g}(X)$ are $\ell^{k - 1}$ and $1$. And if $g \in M_k(N, \CC)$ is cuspidal, and one root of $P_{\ell, g}(X)$ is $\pm \ell^{\frac{k}{2} - 1}$, then the other root must be $\pm \ell^{\frac{k}{2}}$, so that $a_\ell(g) = \pm(\ell + 1)\ell^{\frac{k - 2}{2}}$, which is impossible by \autoref{weilbound}. 
\end{proof}

\subsubsection{Second algebraic notion: kernel of trace} On the other hand, if $B$ is additionally a $\ZZ[\frac{1}{\ell}]$-algebra, then Serre suggests an alternate description of the space of newforms of level $\ell$.

\begin{myprop}[Serre {\cite[\S3.1(c), remarque (3)]{SerreZetaPadic}}]\label{newtrtr}
If $B$ is a characteristic-zero $\ZZ[\frac{1}{\ell}]$-domain, $$M(N\ell, B)^{\ellnew} = \ker \Tr \cap \ker \Tr W_\ell.$$

\end{myprop}

\begin{proof}
Since $B$ is flat over $\ZZ[\frac{1}{\ell}]$, we may replace $\ZZ$ by $\ZZ[\frac{1}{\ell}]$ in the beginning of the proof of \autoref{newkerS} to see that it suffices to establish this in a single weight $k$ for $B = \CC$.
Since both $\Tr$ and $W_\ell$ commute with Hecke operators prime to $\ell$, it suffices to consider separately the one-dimensional eigenspaces spanned by $\ell$-new eigenforms and the two-dimensional $\ell$-old eigenspaces coming from eigenforms of level $N$. If $f \in M_k(N\ell, \CC)$ is $\ell$-new eigen, then both $\Tr f$ and $\Tr W_\ell f$ are forms of level $N$ with the same eigenvalues away from $\ell$ as $f$, which is impossible by \cite[Lemma 23]{atkinlehner}.  Therefore both $\Tr f = 0$ and $\Tr W_\ell f = 0$, so that $\ker \Tr \cap \ker \Tr W_\ell$ does indeed contain $M(N\ell, B)^{\ellnew}$. For the reverse containment, if $f$ is in $M_k(N\ell, \CC)^{\ellold}$, then it suffices to consider to $f$ contained in the two-dimensional span of $g$ and $W_\ell (g)$ for some eigenform $g \in M_k(N, \CC)$. From the identities in \autoref{tracesec}, the operators $\Tr$ and $\Tr W_\ell$, on the ordered basis $\{g, W_\ell(g)\}$ of the $\ell$-old subspace of $M_k(N \ell, \CC)$ associated to $g$, have matrix form
$$\Tr  = \begin{pmatrix} \ell + 1 & \ell\, a_\ell(g) \\ 0 & 0\end{pmatrix} \qquad
\Tr W_\ell = \begin{pmatrix} \ell\, a_\ell(g) & (\ell + 1)\ell^k \\ 0 & 0 \end{pmatrix}.$$
The kernels of matrices of the form $\begin{psmallmatrix} a & b \\ 0 & 0 \end{psmallmatrix}$ and $\begin{psmallmatrix} c & d \\ 0 & 0 \end{psmallmatrix}$ have a nontrivial intersection if and only if $ad = bc$. In our case that would mean that $a_\ell(g)^2 = (\ell + 1)^2 \ell^{k-2}$, which is again impossible by the Weil bounds (\autoref{weilbound}).
\end{proof}

\subsection{Newforms over any domain: a proposal}

Inspired by the algebraic characterisations of \autoref{newkerS} and \autoref{newtrtr} of newforms in characteristic zero, we make the following two definitions.

\begin{mydef} For any ring $B$ and any Hecke-invariant submodule $C \subset M(N \ell, B)$, let
$$C^{\uellnew} : = \ker \left.\SSS\right|_C \andand C^{\tracenew} : = (\ker \left. \Tr \right|_C) \cap (\ker \left. \Tr W_\ell \right|_C).$$
\end{mydef}

\autoref{newkerS} already establishes that if $B$ is a characteristic zero $\ZZ[\frac{1}{\ell}]$-domain and $C$ is $M(N\ell, B)$, then both of these ``$\ellnew$" spaces coincide with $M(N\ell, B)^\ellnew$. In other words, these definitions both extend the Atkin-Lehner analytic notion of $\ell$-new forms.  The main result of this section is to show that on cuspforms, these two definitions coincide  for more general $B$ as well.

\begin{mythm}\label{newformthm}
For any $\ZZ[\frac{1}{\ell}]$-domain $B$, we have $$S(N\ell, B)^\uellnew = S(N \ell, B)^{\tracenew}.$$
\end{mythm}

To prove \autoref{newformthm}, we first establish $\big(S(N\ell, B)^\ellold\big)^\uellnew = \big(S(N \ell, B)^\ellold\big)^{\tracenew}$:

\begin{myprop}\label{keyker}
Let $B$ be a $\ZZ[\frac{1}{\ell}]$-algebra. Suppose $f, g$ in $S_k(N, B)$ for some weight~$k$. Then the following are equivalent.
\begin{enumerate}
\item \label{kerSitem} $f + W_\ell (g) \in \ker \SSS$
\item \label{kertrtritem} $f + W_\ell (g) \in \ker \Tr \cap \ker \Tr W_\ell$
\item \label{lambdaitem} $\ell\, T_\ell\, f = -(\ell + 1) \wt_\ell g$ and $\ell\,T_\ell\, g = -(\ell + 1) f$
\end{enumerate}
\end{myprop}

\begin{remark}
  \
  \begin{enumerate}
\item\label{welloldnew}
\autoref{keyker} may be rewritten more symmetrically in terms of $w_\ell$, the involution-normalized Atkin-Lehner operator on $S_k(N\ell, B)$. Namely, let $\lambda_k = -(\ell + 1) \ell^\frac{k-2}{2}.$ Then the claim of the proposition is that
\begin{align*}
\big(S_k(N\ell, B)^\ellold\big)^\uellnew &= \big(S_k(N \ell, B)^\ellold\big)^{\tracenew} \\
 & = \{f + w_\ell\, g: f, g \in S_k(N, B) \mbox{ s.t. } T_\ell f = \lambda_k g \mbox{ and } T_\ell g = \lambda_k f \}.
 \end{align*}
The constant $\lambda_k$ appears in connection with level-raising theorems of Ribet \cite{ribet:levelraising} and Diamond \cite{diamond:levelraising}. See also \autoref{levelraisingsec} for more details.
\item From the proof \autoref{keyker} below, it is clear that the conclusions hold for any $f, g \in S(N, B)$ as long as we assume that $B$ is $\ZZ[\frac{1}{\ell}]$-domain.
\item \autoref{keyker} does not hold as stated for $M_k(N, B)$ if $B$ has characteristic $p$. For example, if $\ell^{k-2} \equiv 1 \cmod{p}$ but $\ell \neq -1 \cmod{p}$ (say, if $\ell \equiv 1 \cmod{p}$ but $p \neq 2$), then $f  := 1 \in M_{p-1}(N, B)$ is in the kernel of $\SSS$ but is not in the kernel of $\Tr$.%
\end{enumerate}
\end{remark}

\begin{proof}[Proof of \autoref{keyker}]
 We use the identities from \autoref{tracesec} repeatedly, including the fact that for $f \in M(N, B)$, we have 
$\ell\, U_\ell f = \ell\, T_\ell f -  W_\ell f$
and
$U_\ell\, W_\ell f = U_\ell \big(\wt_\ell f(q^\ell)\big) = \wt_\ell f$.
We first show that \eqref{kerSitem}~$\iff$~\eqref{lambdaitem}.
Let $f \in S_k(N,B)$. On one hand we have
\begin{align*}
\SSS f & = \ell^2\,U_\ell^2 f - \wt_\ell f  =
\ell\, U_\ell (\ell \,T_\ell f - W_\ell f )- \wt_\ell f \\
& = \ell\, T_\ell (\ell\, T_\ell f) - W_\ell (\ell\, T_\ell f) - \ell\, U_\ell W_\ell f - \wt_\ell f\\
& = \ell^2\, T_\ell^2 f - (\ell + 1) \wt_\ell f  - W_\ell \,\ell\, T_\ell f
\end{align*}
and
\begin{align*}
\SSS\, W_\ell\, g  & = \ell^2\, U_\ell^2  W_\ell\, g - \wt_\ell W_\ell\, g  = \ell^2 U_\ell \wt_\ell g - W_\ell \wt_\ell  g \\
& =  \ell^2\, T_\ell  \wt_\ell g - \ell\, W_\ell \wt_\ell\, g - W_\ell \wt_\ell \, g = \ell^2 \,T_\ell \wt_\ell g  - (\ell + 1) W_\ell \wt_\ell g.
\end{align*}

From \autoref{oldintk}, the intersection of $S_k(N, B)$ and $W_\ell S_k(N, B)$ inside $M_k(N\ell, B)$ is trivial. So $\SSS(f + W_\ell g) = 0$ if and only if
\begin{align*}
0 &= \SSS(f + W_\ell g)\\
 &= \big(\ell^2\, T_\ell^2 f - (\ell + 1)\wt_\ell f + \ell^2 T_\ell \wt_\ell g\big) -  W_\ell \big(\ell\, T_\ell f   + (\ell + 1) \wt_\ell g\big),
\end{align*}
which holds if and only if
$$\ell^2\, T_\ell^2 f - (\ell + 1)\wt_\ell f + \ell^2 T_\ell \wt_\ell g = 0 \qquad \mbox{and} \qquad \ell\, T_\ell f   + (\ell + 1) \wt_\ell g = 0.$$
The second equation reduces to
\begin{equation}\label{tellf}
\ell\, T_\ell f =  -(\ell + 1) \wt_\ell g.
\end{equation}
Inserting this into the first equation, combining like terms, and eliminating $\wt_\ell$ reveals
\begin{equation}\label{tellg}
\ell\, T_\ell g =  -(\ell + 1)f,
\end{equation}
as required.

For \eqref{kertrtritem} $\iff$ \eqref{lambdaitem}, we recall that for $f \in M_k(N,B)$,
$$\Tr f = (\ell + 1) f \qquad \mbox{and} \qquad \Tr W_\ell f = \ell \, T_\ell f.$$
Therefore $\Tr(f + W_\ell g) = 0$ $\iff$ $(\ell + 1) f + \ell \, T_\ell g = 0$ $\iff$ $\ell\,T_\ell\, g  = -(\ell + 1) f$. Symmetrically, $0 = \Tr W_\ell (f + W_\ell g) = \Tr (\wt_\ell g + W_\ell f)$ $\iff$ $\ell\, T_\ell\, f = - (\ell + 1) \wt_\ell g$.
\end{proof}

\begin{proof}[Proof of \autoref{newformthm}]
If $B$ has characteristic zero, then this statement is already known (\autoref{newkerS} \& \autoref{newtrtr}), but we prove it again without using the Weil bound. As in the proof of \autoref{newkerS}, we may assume that we are in a single weight $k$ and that $B = \CC$, and note that each one-dimensional $\ell$-new eigenspace is annihilated by all three operators $\SSS$, $\Tr$, and $\Tr W_\ell$. Now \autoref{keyker} establishes the desired statement for each two-dimensional $\ell$-old away-from-$\ell$ Hecke eigenspace and completes the  proof.

If $B$ has characteristic $p$, then we may assume that $B = \FF_p$ and again as in the proof of \autoref{oldint} work in a single weight $k$. We will have to distinguish between coefficients in $\ZZ_p$ and quotients, so for any ring $B$, write $X^B$ for the operator $X$ acting on $S_k(N \ell, B)$.

Take $f \in S_k(N\ell, \FF_p)$. Then there exist integral forms $\tilde f^\ellnew$ and $\tilde f^\ellold$ in $S_k(N\ell, \ZZ_p)^{\ellnew}$ and $S_k(N\ell, \ZZ_p)^{\ellold}$, respectively, and a $b \in \ZZ_{\geq 0}$ so that $f$ is the mod-$p$ reduction of
$$\tilde f = p^{-b}(\tilde f^\new + \tilde f^\old) \in S_k(N\ell, \ZZ_p).$$
Suppose now that $f \in \ker \SSS^{\FF_p}$, so that $\SSS^{\ZZ_p}(\tilde f)$ is in $p \ZZ_p \lb q \rb$. Since $\SSS^{\ZZ_p}(\tilde f^\new) = 0$ we have $\SSS^{\ZZ_p}(\tilde f) = p^{-b} \SSS^{\ZZ_p}(\tilde f^\old)$. In other words, the form $f^\old$ is in $\ker \SSS^{\ZZ/p^{b + 1}\ZZ}$, where $ f^\old \in S_k(N\ell,\ZZ/p^{b + 1} \ZZ)$ is the image of $\tilde f^\old$ under the reduction-mod-$p^{b + 1}$ map.
By \autoref{keyker}, $f^\old$ is in $\ker (\Tr)^{\ZZ/p^{b + 1} \ZZ} \cap \ker (\Tr W_\ell)^{\ZZ/p^{b + 1} \ZZ}$. By lifting back up to characteristic zero, we see that both $\Tr^{\ZZ_p} (\tilde f^\old)$ and $(\Tr W_\ell)^{\ZZ_p}(\tilde f^\old)$ are in $p^{b+1}\ZZ_p\lb q \rb$.

As $(\Tr)^{\ZZ_p} (\tilde f^\new)=(\Tr W_\ell)^{\ZZ_p} (\tilde f^\new)=0$, we get that both $\Tr^{\ZZ_p} (\tilde f)$ and $(\Tr W_\ell)^{\ZZ_p} (\tilde f)$ are in $p\ZZ_p\lb q \rb$. Therefore, $\Tr^{\FF_p} (f) \equiv \Tr^{\ZZ_p} (\tilde f) \equiv 0 \cmod{p}$ and
\begin{equation*}
  (\Tr W_\ell)^{\FF_p} (f) \equiv (\Tr W_\ell)^{\ZZ_p} (\tilde f) \equiv 0 \cmod{p}.
\end{equation*}
Hence $f$ is in $\ker (\Tr)^{\FF_p} \cap \ker (\Tr W)^{\FF_p}$. Reverse all steps for the reverse containment.

\end{proof}

In light of \autoref{newformthm}, we introduce the following definition:
\begin{mydef}
If $B$ is any $\ZZ[\frac{1}{\ell}]$-algebra, the submodule of $\ell$-new cuspforms of weight $k$  is
$$S_k(N\ell, B)^\ellnew : = S_k(N\ell, B)^\uellnew = S_k(N\ell, B)^\tracenew.$$
If $B$ is any $\ZZ[\frac{1}{\ell}]$-domain, the submodule of $\ell$-new cuspforms of all weights is
$$S(N\ell, B)^\ellnew : = S(N\ell, B)^\uellnew = S(N\ell, B)^\tracenew.$$
\end{mydef}

We will also use the notation $M(N\ell, B)^\ellnew := M(N\ell, B)^\tracenew$. Observe that the space of $\ell$-new forms is stable under $W_\ell$.

\section{Interactions between $\ell$-old and $\ell$-new spaces mod $p$}\label{ellnewelloldsec}

In characteristic zero, spaces of $\ell$-new and $\ell$-old forms are disjoint. This fails in characteristic $p$ because of congruences between $\ell$-new and $\ell$-old forms. A related phenomenon: over a field of characteristic zero, $\ell$-new and $\ell$-old forms together span the space of forms of level $N\ell$. This already fails over a ring like $\ZZ_p$, again because of congruences between $\ell$-new and $\ell$-old forms. A guiding scenario: if $f \in S_k(N\ell, \ZZ_p)^\ellnew$ is nonzero modulo $p$ but congruent to $g \in S_k(N\ell, \ZZ_p)^\ellold$ modulo $p$ but not modulo $p^2$, then $\frac{1}{p}({f - g})$ is in $S_k(N\ell, \ZZ_p)$ but not in $S_k(N\ell, \ZZ_p)^\ellnew \oplus S_k(N\ell, \ZZ_p)^\ellold$, and the (nonzero) reduction $\bar f$ of $f$ modulo $p$ is in $S_k(N\ell, \FF_p)^\ellnew \cap S_k(N\ell, \FF_p)^\ellold$.
\begin{example}
Take $N = 5$, $\ell = 3$, $p= 7$, $k = 4$. There is only one cuspform at level $N$, namely,
$f = q - 4q^{2} + 2q^{3} + 8q^{4} - 5q^{5} - 8q^{6} + 6q^{7} - 23q^{9} + O(q^{10}) \in S_4(5, \ZZ_7).$
In level $N\ell$, there are two newforms, forming a basis of $S_4(15, \ZZ_7)$ (but not over $\ZZ$, as they are congruent modulo $2$):
\begin{align*} a &= q + q^{2} + 3q^{3} - 7q^{4} + 5q^{5} + 3q^{6} - 24q^{7} - 15q^{8} + 9q^{9} + O(q^{10})\\
 b & = q + 3q^{2} - 3q^{3} + q^{4} - 5q^{5} - 9q^{6} + 20q^{7} - 21q^{8} + 9q^{9} - O(q^{10})
\end{align*}
\end{example}
One can check that $b \equiv f + 2f(q^3)$ modulo $7$\footnote{Indeed, the level-raising condition for $f$ at $3$ modulo $7$ is satisfied, so that the existence of such a congruence is guaranteed by Diamond \cite{diamond:levelraising}. See also \autoref{levelraisingsec}.} and that $\frac{1}{7}(f + \frac{2}{81} W_3 f -  b)$ is in $S_4(15, \ZZ_7)$ but not in $S_4(15, \ZZ_7)^{3\myhyphen\old} \oplus S_4(15, \ZZ_7)^{3\myhyphen\new}$. Modulo $7$, we likewise find $\bar b$ in $S_4(15, \FF_7)^{3\myhyphen\old} \cap S_4(15, \FF_7)^{3\myhyphen\new}.$

In this section, we describe the intersection of the $\ell$-old and the $\ell$-new subspaces modulo $p$ and comment on the failure of these to span the whole level-$N\ell$ space. We will fix a prime $p$ and work with $B = \FF_p$ or a finite extension, suppressing $B$ from notation. We start with the following corollary to \autoref{keyker} and the first remark following:

\begin{mycor}\label{keykerinter}
\begin{enumerate} \leavevmode
\item $S(N\ell)^{\ellold} \cap S(N\ell)^{\ellnew}$
$$ = \big\{f + W_\ell (g): \ f, g \in S(N),\  \ell\, T_\ell f  = -(\ell + 1) \wt_\ell g,\ \ell\, T_\ell g = -(\ell + 1) f\big\}.$$
\item If $p \neq 2$, then in fixed weight $k$ with $\lambda_k = -(\ell + 1) \ell^\frac{k-2}{2}$, we have
$$S_k(N\ell)^{\ellold} \cap S_k(N\ell)^{\ellnew} = V_{\lambda_k}^+ \oplus V_{-\lambda_k}^-,$$ where, for $\alpha \in \FF_p$, we write
$V_\alpha^\pm := \left\{f \pm w_\ell f: f \in \ker\left.(T_\ell - \alpha)\right|_{S_k(N)}\right\}.$
\end{enumerate}
\end{mycor}

\begin{proof}
The first part of the corollary follows directly from \autoref{keyker}. For the second part, first observe that $V_{\lambda_k}^+ \oplus V_{-\lambda_k}^-  \subset S_k(N\ell)^{\ellold} \cap S_k(N\ell)^{\ellnew}$ by the first remark after \autoref{keyker}. As we are assuming $p \neq 2$, we can write $f+w_{\ell}(g) = \frac{1}{2}\big((f+g + w_{\ell}(f+g))+(f-g - w_{\ell}(f-g))\big)$. If $T_{\ell} f =\lambda_k g$ and $T_\ell g = \lambda_k f$, then $f+g+w_{\ell}(f+g) \in V_{\lambda_k}^+$ and $f-g-w_{\ell}(f-g) \in V_{-\lambda_k}^-$. The corollary now follows directly from the first remark after \autoref{keyker}.
\end{proof}

To offer a more detailed analysis, we will pass to generalized Hecke eigenspaces. In \autoref{heckealgsetup} we recall definitions and notations for mod-$p$ big Hecke algebras. And in \autoref{levelraisingsec} we state our conclusions on the intersection of $\ell$-old and $\ell$-new subspaces in characteristic $p$.

\subsection{The Hecke algebra acting on modular forms mod $p$}\label{heckealgsetup}
In this section, we briefly recall the construction of the big mod-$p$ Hecke algebra acting on $M(N) = M(N, \FF)$. For more details, see \cite[1.2]{BK} or \cite[2.3--2.5]{medved:thesis} for the construction for $N = 1$,  \cite[Section 1]{deo:hecke} for general $N$.

We work over $B = \FF$, a finite extension of $\FF_p$. For any level $N$, let $A(N) = A(N, \FF)$ be the closed Hecke algebra topologically generated inside $\eend_\FF\!\big(M(N)\big)$ by the action of Hecke operators $T_n$ for $n$ prime to $N p$ under the compact-open topology on $\eend_\FF\!\big(M(N)\big)$ induced by the discrete topology on $M(N)$. We write $A(N) = \hecke(M(N))$ for this construction. This is the \emph{big} \emph{shallow} Hecke algebra acting on the space of modular forms of level $N$ modulo $p$, the only kind of Hecke algebra we study here.\footnote{One can also consider the \emph{big partially full} Hecke algebra $A(N)^\pf$, topologically generated inside $\eend_{\FF}\!\big(M(N)\big)$ by the action of $T_n$ for all $(n, Np) = 1$ as well as $U_\ell$ for $\ell \mid N$, and the \emph{big full} Hecke algebra $A(N)^\full$, which also includes the action of $U_p$. Many authors also consider the ``smaller" algebras $A_k(N)$, $A_k(N)^\pf$, $A_k(N)^\full$ acting on forms in a single weight.}

One can show that $A(N)$ is a complete noetherian semilocal ring that factors into a product of its localizations at its maximal ideals, which by Deligne and Serre reciprocity (formerly Serre's conjecture) correspond to Galois orbits of odd dimension\nobreakdash-$2$ Chenevier pseudorepresentations $(t, d): G_{\QQ, N p} \to \bar \FF_p$, where $d = \omega_p^{\kappa-1}$ for some $\kappa \in \ZZ/(p-1) \ZZ$. Here $\omega_p$ is the mod-$p$ cyclotomic character, and $G_{\QQ, Np}$ is the Galois group $\gal(\QQ_{Np}/\QQ)$, where $\QQ_{Np}$ is the maximal extension of $\QQ$ unramified outside the support of $Np\infty$. Since the $d$ in each pseudorepresentation is entirely determined by $t$ in this $\Gamma_0(N)$ setting (indeed, if $p > 2$ we have $d(g) = \frac{t(g)^2 - t(g^2)}{2}$ for any $g \in G_{\QQ, Np}$; and if $p = 2$ then $d = 1$), we will frequently suppress it from notation. For more on Chenevier pseudorepresentations see \cite{chen} or \cite[1.4]{BK}.  If we assume that $\FF$ is large enough to contain all the finitely many Hecke eigenvalue systems appearing in $M(N)$, then the Galois orbits become trivial; from now on we assume that this is done.

Let $K(N) \subset M(N)$ be the kernel of the $U_p$ operator. Since $U_p$ in characteristic $p$ is a left inverse of the raising to the $p^{\rm th}$ power operator $V_p$, given any form $f \in M(N\ell)$ the form $g = (1 - V_p U_p)f$ has the property that $a_n(g) = a_n(f)$ unless $p \mid n$, in which case $a_n(g) = 0$. Therefore $K(N)$ is a nontrivial subspace of $M(N)$. Further, since $U_p$ preserves the grading from \eqref{gradingdecomp}, we can set $K(N)^k : = K(N) \cap M(N)^k$ for $k \in \ZZ/(p-1)\ZZ$ and then $K(N) = \bigoplus_k K(N)^k$. One can show that $A(N)$ acts faithfully on $K(N)$, so that $A(N)$ is also $\hecke(K(N))$. Studying this smaller space eliminates minor complications caused by the behavior of our Hecke eigensystems at $p$.

For $\kappa \in \ZZ/(p-1)\ZZ$, let $$\PS_\kappa(N) := \{(t, d): G_{\QQ, N p} \to \bar \FF_p\ \mbox{odd Chenevier pseudorepresentation with $d = \omega_p^{\kappa-1}$}\},$$ and let $\PS(N) = \bigcup_{\kappa \in \ZZ/(p-1)\ZZ} \PS_\kappa(N)$.

By the remarks above, $\PS(N)$ corresponds to the set of maximal ideals of $A(N)$. Let $A(N)_t$ be the localization of $A(N)$ at the maximal ideal corresponding to $t \in \PS(N)$. This is a complete local noetherian ring, and we have a decomposition
$$A(N) = \prod_{t \in \PS(N)} A(N)_t.$$
The factorization of $A(N)$ leads to a splitting of $M(N)$ and $K(N)$ into generalized eigenspaces for $t \in \PS(N)$, refining the gradings on $M(N)$ and $K(N)$:
$$M(N)^\kappa = \bigoplus_{t \in \PS_\kappa(N)} M(N)_t \andand K(N)^\kappa = \bigoplus_{t \in \PS_\kappa(N)} K(N)_t.$$

\subsection{$\ell$-old and $\ell$-new forms restricted to eigencomponents}\label{levelraisingsec}

We now return to working with modular forms of level $N \ell$, where $\ell$ is a prime not dividing $Np$. Recall that we work over $B = \FF$, an extension of $\FF_p$ containing all of the Hecke eigensystems appearing in $M(N\ell) = M(N\ell, \FF)$.

Since the operators $\Tr$, $W_\ell$, $\SSS$ used to define the $\ell$-old and $\ell$-new subspaces of $M(N\ell)$, commute with Hecke operators away from $\ell$, the spaces $M(N\ell)^\ellnew$ and $M(N\ell)^\ellold$ also decompose into generalized eigenspaces for the various $t \in \PS(N\ell)$. For a Hecke module $C \subset M(N\ell)$, write $C_t : = C \cap M(N \ell)_t$, so that we define $S(N\ell)_t$, $S(N\ell)_t^{\ellold}$ and $S(N\ell)_t^{\ellnew}$.
\begin{mythm}\label{levelraising}
Fix $\kappa \in 2 \ZZ/(p-1)\ZZ$ (or $\kappa = 0$ if $p= 2$) and $t \in \PS_\kappa(N\ell)$. For $k$ even with $k \equiv \kappa \cmod{(p-1)}$, let $\lambda_k$ be the image of $-(\ell + 1) \ell^\frac{k-2}{2}$ in $\FF_p$. Note that the set $\{\pm\lambda_k\}$ depends only on $\kappa$.

\begin{enumerate}
\item If $t \in \PS_{\kappa}(N\ell) - \PS_{\kappa}(N)$ (that is, any representation carrying $t$ is ramified at $\ell$), then no forms of level $N$ carry this eigensystem.  Therefore $M(N\ell)^{\ellold}_t = 0$ and hence $M(N\ell)_t = M(N\ell)_t^{\ellnew}$.
\item\label{levelN} Otherwise, $t \in \PS_\kappa(N)$, and we are in one of two situations:
\begin{enumerate}
\item\label{levelraisingeasy} If $t(\frob_\ell) \neq \pm \lambda_k$, then $M(N\ell)_t^{\ellnew} = 0$, and therefore $M(N\ell)_t = M(N\ell)_t^{\ellold}$.
\item \label{levelraisinginter} If $t(\frob_\ell) = \pm \lambda_k$, then all three of $M(N\ell)_t^{\ellold}$, $M(N\ell)_t^{\ellnew}$, and\\ $M(N\ell)_t^{\ellnew} \cap M(N\ell)_t^{\ellold}$ are nonzero. Moreover:
\begin{enumerate}
\item\label{zerolam} If $\lambda_k = 0$, then, writing $\ker T_\ell$ for $\ker \left.T_\ell\right|_{S(N)_t}$, we have
$$S(N\ell)_t^{\ellold} \cap S(N\ell)_t^{\ellnew} = \big(\ker T_\ell \big) \oplus W_\ell \big(\ker T_\ell\big).$$
\item\label{nonzerolam} If $\lambda_k \neq 0$, let $\eps_k = \pm 1$ be determined by $t(\frob_\ell) = \eps_k \lambda_k$.
Then
$$S(N\ell)_t^{\ellold} \cap S(N\ell)_t^{\ellnew}
 = \big\{f - \eps_k w_\ell f :  f \in \ker\left.(T_\ell -  \eps_k\lambda_k)\right|_{S(N)_t}\big\}.$$

\end{enumerate}
\end{enumerate}
\end{enumerate}

\end{mythm}

In part \eqref{nonzerolam}, note that $\eps_k w_\ell$ depends only on $\kappa$, not on $k$ (in other words, $\eps_k w_\ell$ is well defined on $S(N\ell)_t$). It also straightforward to see that $\eps_k w_\ell = ({\eps_k \lambda_k})\ell (\ell + 1)^{-1}\wt_\ell^{-1}W_\ell$.
The statements of \autoref{levelraising} dovetail nicely with the level-raising results \cite{ribet:levelraising, diamond:levelraising}: if $f$ is an integral eigenform of level $N$ and weight $k$ whose mod-$p$ representation is absolutely irreducible, then there is another eigenform of level $N\ell$ congruent modulo $p$ to $f$
 (away from $N\ell p $) if and only if $a_\ell(f)^2 \equiv \lambda_k^2$ modulo $p$.
For a level-$N$ pseudorepresentation $t$ mod $p$, we will say that the \emph{level-raising condition is satisfied for $(t, \ell)$} if $t(\frob_\ell) = \pm \lambda_k$.

\begin{proof}[Proof of \autoref{levelraising}]
If $t$ does not factor through $G_{\QQ, Np}$, then there are no $\ell$-old eigenforms and every form is $\ell$-new: this will be true mod $p$ because it is true over $\bar\ZZ_p$. So assume $t \in \PS_
\kappa(N)$, carried by some eigenform $f' \in S(N)$. If $M(N\ell)_t^{\ellnew} = \ker \left.\SSS\right|_{M(N\ell)_t}$ is nonzero, then it contains an eigenform $g$, which by assumption is also an eigenform for $U_\ell$ with eigenvalue $\pm\ell^\frac{k-2}{2}$. Since $g$ is $\ell$-old (more precisely, since $g$ can be lifted to an $\ell$-old eigenform in characteristic zero by the Deligne-Serre lifting lemma), there exists an eigenform $f \in M_k(N)$ of some weight $k$ such that $g$ is contained in the subspace $V_{\ell,f}$ of $M(N\ell)$ generated by $f$ and $W_{\ell}(f)$, and the characteristic polynomial of $U_{\ell}$ acting on $V_{\ell,f}$ is $X^2 - a_\ell(f) X+ \ell^{k-1} = X^2 - t(\frob_\ell)X + \ell^{k-1}$. Since one root of this polynomial is $\pm \ell^\frac{k-2}{2}$ (that is, the $U_\ell$-eigenvalue of $g$), the other root is $\pm \ell(\ell^\frac{k-2}{2})$, so that $t(\frob_\ell) = \pm (\ell + 1) \ell^\frac{k-2}{2} = \mp \lambda_k$.
This proves \eqref{levelraisingeasy}.

For \eqref{levelraisinginter}: if $\lambda_k = 0$, then remark \eqref{welloldnew} after \autoref{keyker} restricted to $S(N\ell)_t$ gives us $f + W_\ell g \in S(N\ell)^{\ellold}_t \cap S(N\ell)^{\ellnew}_t$ if and only if $f$ and $g$ are in $S(N)_t$ and killed by $T_\ell$. If $\lambda_k$ is nonzero (so $p \neq 2$), then only one of $\pm \lambda_k$, namely $\eps_k \lambda_k$, appears as a $T_\ell$-eigenvalue in $S(N)_t$. In particular, from the formulation in \autoref{keykerinter}, we see that $f+ w_\ell g \in S_k(N\ell)^{\ellold}_t \cap S_k(N\ell)^{\ellnew}_t$ if and only if $f$ is in the kernel of $T_\ell - \eps_k \lambda_k$ and $g = \eps_k f$.
But any $f$ and $g$ in $S(N)_t$ appear together in some weight $k$.
\end{proof}

\subsection{The span of $\ell$-old and $\ell$-new forms}  If $B$ is a field of characteristic zero, then we always have $S(N\ell,B)^\ellnew  \oplus S(N\ell,B)^\ellold=S(N\ell,B)$. But the analogous statement fails already for $B = \ZZ_p$, as $S(N\ell,B)^\ellnew  \oplus S(N\ell,B)^\ellold$ may miss congruences between $\ell$-old and $\ell$-new forms. For $B = \FF_p$ and extensions, we no longer expect a direct sum in general, but we may still ask whether $\ell$-old and $\ell$-new forms together span all cuspforms.
To illuminate the behavior most effectively, we restrict to a generalized eigenspace for some $t \in \PS(N\ell)$.

To this end, fix $t$, let $\FF$ be an extension of $\FF_p$ containing its values, and let $\OO := W(\FF)$, the unique unramified extension of $\ZZ_p$ with residue field $\FF$. We have defined $S(N\ell, \FF)_t$ as the set of generalized eigenforms in $S(N\ell, \FF)$ for the (shallow) Hecke eigensystem carried by~$t$. We define $S(N\ell, \OO)_t$ as the subspace of $S(N\ell, \OO)$ consisting of linear combinations of eigenforms whose corresponding shallow Hecke eigensystem is a lift of $t$. Unlike in characteristic $p$, it will no longer be true that every eigensystem is defined over $\OO$, but if $\FF$ is large enough to contain the values of all the elements of $\PS(N\ell)$, then it is still true that $S(N\ell, \OO)$ splits as a direct sum of all its generalized $t$-eigenspaces  $S(N\ell, \OO) = \bigoplus_{t \in \PS(N\ell)} S(N\ell, \OO)_t$. See \cite[Section 1]{deo:hecke} for details. Similarly, we define $S(N\ell, \OO)_t^\ellold$ and $S(N\ell, \OO)^\ellnew_t$.

\begin{myprop}
With $t$, $\FF$, $\OO$ as above, the following are equivalent:
\begin{enumerate}
\item\label{SSSsurj} The action of $\SSS$ on ${S(N\ell,\OO)_t^{\ellold}}$ is surjective.
\item\label{nontrivint} The intersection $S(N\ell,\FF)_t^{\ellnew} \cap S(N\ell,\FF)_t^{\ellold}$ is trivial.
\item\label{FFsum} $S(N\ell,\FF)_t = S(N\ell,\FF)_t^{\ellnew} \oplus S(N\ell,\FF)_t^{\ellold}.$
\item\label{satisfylevelraise} Either $t$ is new at $\ell$, or $(t, \ell)$ does not satisfy the level-raising condition.
\end{enumerate}
If these equivalent conditions hold, then we additionally have
\begin{enumerate}[resume]
\item\label{OOsum} $S(N\ell,\OO)_t = S(N\ell,\OO)_t^{\ellnew} \oplus S(N\ell,\OO)_t^{\ellold}.$
\end{enumerate}
Finally, if $t$ is absolutely irreducible\footnote{That is, $t$ is not the sum of two characters $G_{\QQ, N\ell p} \to \bar \FF_p$.}, then \eqref{SSSsurj}, \eqref{nontrivint}, \eqref{satisfylevelraise}, \eqref{FFsum} and \eqref{OOsum} are all equivalent.
\end{myprop}

\begin{proof}
The equivalence of \eqref{nontrivint}, \eqref{satisfylevelraise} and \eqref{FFsum} follows from \autoref{levelraising}.

We demonstrate \eqref{SSSsurj}~$\iff$~\eqref{nontrivint}: Since $S(N\ell, \OO)_t^\ellold$ breaks up into a graded sum of its fixed-weight pieces, and since $\SSS$ is weight-preserving, surjectivity on $S(N\ell, \OO)_t^\ellold$ is equivalent to surjectivity on $S_k(N\ell, \OO)^\ellold_t$. By right-exactness of tensoring or Nakayama's lemma (depending on the direction) this last is equivalent to surjectivity on $S_k(N\ell, \FF)^\ellold_t$. This space is a finite-dimensional vector space, so $\SSS$ acts surjectively if and only if it has trivial kernel, which is equivalent by definition to $S_k(N \ell, \FF)^\ellold_t \cap S_k(N\ell, \FF)^\ellnew_t = \{0\}$. Finally trivial intersection in all finite weights $k$ is equivalent to trivial intersection of $S(N\ell, \FF)^\ellnew_t$ and $S(N\ell, \FF)^\ellold_t$.

Now \eqref{SSSsurj}~$\implies$~\eqref{OOsum}: The surjectivity on $S(N\ell, \OO)_t^\ellold$ implies the that for both $B = \OO$ and $B = \FF$, the following sequence is split exact.
$$0 \to S(N\ell, B)_t^\ellnew \to S(N\ell, B)_t \stackrel{\SSS}\to S(N\ell, B)_t^\ellold \to 0,$$
which means that $S(N\ell, B)_t = S(N\ell, B)_t^\ellold \oplus S(N\ell, B)_t^\ellnew$.

Finally, if $t$ is absolutely irreducible, then the level-raising theorems \cite{ribet:levelraising, diamond:levelraising} hold. Therefore if $t \in \PS(N)$ and $(t, \ell)$ satisfies the level-raising condition, then there exists an $\ell$-new form congruent to an $\ell$-old form (over some extension of $\OO$), which implies that
\begin{equation*}
  S(N\ell, \OO)_t \supsetneq S(N\ell,\OO)_t^{\ellnew} \oplus S(N\ell,\OO)_t^{\ellold}.
\end{equation*}
\end{proof}

\begin{question} Is it always true that $S(N\ell,\FF_p)_t^\ellnew+ S(N\ell,\FF_p)_t^\ellold=S(N\ell,\FF_p)_t$? A positive answer would furnish additional support for the present definition of $\ell$-new forms.
\end{question}
\section{Hecke-stable filtrations mod $p$}\label{monskysec}

In this section we describe a filtration for the space of modular forms of level $N\ell$ modulo~$p$, and compare it to the filtration described by Monsky in \cite{monsky:level3fil, monsky:level5fil}, which appears if $\ell \equiv -1$ modulo $p$. We assume that $B = \FF$, a finite extension of $\FF_p$ big enough to contain all mod-$p$ eigensystems, throughout, and suppress $B$ from notation.

\subsection{The standard filtration (after Paul Monsky)}\label{standardfilsec}
For simplicity, we will restrict to the kernel of the $U_p$ operator $K(N\ell) \subset M(N\ell)$, where formulas are simpler but no Hecke eigensystem information is lost. See also \autoref{heckealgsetup} and \autoref{levelraisingsec} for additional notation.  Then $K(N\ell)$ contains two subspaces
$$K(N\ell)^{\ellold} = K(N) \oplus W_\ell K(N) \andand K(N\ell)^{\ellnew} := \ker\SSS = \ker \Tr  \cap \Tr W_\ell  .$$
Here the action of all operators is restricted to $K(N\ell)$, so that $\ker \SSS = \ker \left.\SSS\right|_{K(N\ell)}$, etc.

The Hecke algebra $A(N\ell) = \hecke\big(K(N\ell)\big)$ has quotients $A(N\ell)^{\ellnew} := \hecke\big(K(N\ell)^{\ellnew}\big)$ and  $$A(N)^\ellold := \hecke\big(K(N\ell)^{\ellold}\big) \cong \hecke\big(K(N)\big)= A(N).$$ To study the Hecke structure on $K(N\ell)$ more closely, we consider the following filtration by Hecke-invariant submodules, which we'll call the \emph{standard filtration}:
\begin{equation}
0 \subset K(N\ell)^{\ellnew} \subset \ker \Tr \subset K(N\ell).
\end{equation}
For any $t \in \PS(N\ell)$, we can pass to the sequence on the $t$-eigenspace:
\begin{equation}\label{myfil}
0 \subset K(N\ell)_t^{\ellnew} \subset (\ker \Tr)_t \subset K(N\ell)_t.
\end{equation}

We also consider the following two conditions relative to a pseudorepresentation $t \in \PS(N)$ and a Hecke operator $T \in A(N)_t$.

\begin{itemize}
\item[] {\bf Condition $\cond({t, T})$}: Operator $T \in A(N)_t$ acts surjectively on $K(N)_t$.
\item[] {\bf Condition $\zerodiv(t, T)$}: Element $0 \neq T \in A(N)_t$ is not a zero divisor on $K(N)_t$.
\end{itemize}

Note that $\cond(t, T)$ implies $\zerodiv(t, T)$: suppose $T K(N)_t = K(N)_t$, and suppose there exists $T' \in A(N)_t$ with $T' T = 0$. Then $T'$ annihilates $K(N)_t$; since the action of $A(N)_t$ is faithful, we must have $T' = 0$. Both conditions are satisfied if $A(N)_t$ is a regular local $\FF$-algebra of dimension $2$.\footnote{It's not unreasonable to expect that this is always the case for $N = 1$. No counterexamples are known; for reducible $t \in \PS(1)$, Vandiver's conjecture implies that $A(1)_t$ is a regular local ring of dimension~$2$: see \cite[\S 10]{BK}. 
}
See section \autoref{tellsurjproof} below for more details.

We are now ready to analyze the standard filtration \eqref{myfil}.

If $t \in \PS(N\ell) \backslash \PS(N)$, then $K(N\ell)_t = K(N\ell)_t^{\ellnew}$, so that the filtration stabilizes; clearly then $A(N\ell)_t = A(N\ell)^{\ellnew}_t$. For the rest of this section, assume that $t \in \PS_\kappa(N)$. Recall that $\lambda_k$ is the image of $ -(\ell + 1) \ell^{\frac{k-2}{2}}$ in $\FF_p$.

\begin{myprop}\label{standardfil} Suppose that $t \in \PS_\kappa(N)$.

\begin{enumerate}
\item\label{tr} If
$\left.\begin{cases} \mbox{EITHER $\ell \not\equiv -1$ modulo $p$,}\\
\mbox{OR $\ell \equiv -1$ modulo $p$ and $\cond(t, T_\ell)$ holds}
\end{cases}\right\},$
then $$K(N\ell)_t/(\ker \Tr) \cong K(N)_t.$$
\item\label{trw}
If $\left.\begin{cases} \mbox{EITHER $\ell \not\equiv -1$ mod $p$ and $\cond(t, T_\ell^2 - \lambda_k^2)$ holds}\\
\mbox{OR  $\ell \equiv -1$ mod $p$ and $\cond(t, T_\ell)$ holds}
\end{cases}\right\},$
then $$(\ker \Tr)/K(N\ell)^{\ellnew}_t \cong K(N)_t.$$
\end{enumerate}
\end{myprop}

In other words, under regularity conditions on $A(N)_t$, the Hecke algebras acting on the graded pieces of the standard filtration are one copy of $A(N\ell)^\ellnew_t$ and two copies of $A(N\ell)^\ellold_t$. Note that $K(N\ell)^\ellnew_t$ and $A(N\ell)^\ellnew_t$ will be zero if the level-raising condition for $(t, \ell)$ is not satisfied.

\begin{proof}
For part \eqref{tr}, we show that under the given conditions, the sequence $$0 \to (\ker \Tr)_t \to K(N\ell)_t \stackrel{\Tr}{\longrightarrow} K(N)_t \to 0$$ is exact. On the left, exactness is by definition. On the right, if $\ell \not\equiv -1$ modulo $p$ then for any $f \in K(N)$ we have $\Tr(f) = (\ell + 1)f$, which spans $\langle f \rangle_\FF$. Otherwise, $\Tr W_\ell f = \ell T_\ell(f)$, so condition $\cond(t, T_\ell)$ suffices.

For part \eqref{trw}, we establish the exactness of
\begin{equation}\label{standardmiddle}
0 \to K(N\ell)_t^{\ellnew} \to (\ker \Tr)_t \stackrel{\Tr W_\ell}{{\longrightarrow}} K(N)_t \to 0.
\end{equation} Again, left exactness holds since $K(N\ell)^{\ellnew} = \ker \Tr \cap \ker \Tr W_\ell$. For right exactness, if $\ell \equiv -1 \mod p$, then $K(N)_t \subset \ker \Tr$, and then $\Tr W_\ell  f =  \ell T_\ell f$ for any $f \in K(N)_t$. Otherwise use the computations of \autoref{keyker} to see that $g = T_\ell f - (\ell+1)/\ell W_\ell f$ is in $\ker \Tr$, and then $\Tr W_\ell (\ell^{-1} g) = (T_\ell^2  - \lambda_k^2) f$.
\end{proof}

\begin{mycor}\label{standardfilcor}
If $t \in \PS_\kappa(N)$ and both $\cond(t, T_\ell)$ and $\cond(t, T_\ell^2 - \lambda_k^2)$ hold, then the graded pieces associated to the standard filtration of $K(N\ell)_t$ are isomorphic to two copies of $K(N)_t$ and one copy of $K(N\ell)_t^{\ellnew}$. The corresponding Hecke algebras are $A(N)_t$, $A(N)_t$, and $A(N\ell)_t^{\ellnew}$.
\end{mycor}

Note that if the level-raising condition for $(t, \ell)$ is not satisfied, then both $K(N\ell)^{\ellnew}_t = 0$ and $A(N\ell)_t^{\ellnew} = 0$; both \autoref{standardfil} and \autoref{standardfilcor} hold.

We can in fact slightly relax the assumptions of \autoref{standardfilcor}:

\begin{myprop}\label{standardfilzerodiv}
If $t \in \PS_\kappa(N)$, and both $\zerodiv(t, T_\ell)$ and $\zerodiv(t, T_\ell^2 - \lambda_k^2)$ hold, then the Hecke algebras on graded pieces of the standard filtration are two copies of $A(N)_t$ and one copy of $A(N\ell)_t^{\ellnew}$.
\end{myprop}

\begin{proof}
From the proof of \autoref{standardfil}, we see that $K(N\ell)_t/(\ker \Tr)_t$ is isomorphic to a Hecke module that sits between $T_\ell K(N)_t$ and $K(N)_t$. If $T_\ell$ is not a zero divisor on $K(N)_t$, then $A(N)_t$ acts faithfully on $T_\ell K(N)_t$: indeed, if any $T \in A(N)_t$ annihilates $T_\ell K(N)_t$, then $T T_\ell$ annihilates $K(N)_t$. Therefore the Hecke algebra on $T_\ell K(N)_t$, and hence on $K(N\ell)_t/(\ker \Tr)_t$, is still $A(N)_t$. The reasoning for the Hecke algebra on $(\ker \Tr)_t/K(N\ell)^{\ellnew}_t$ is analogous.
\end{proof}

\subsection{Connection to the Monsky filtration}

In \cite{monsky:level3fil} and \cite{monsky:level5fil}, Monsky studies $K(N\ell)$ and related Hecke algebras in the case $p = 2$, $N=1$ and $\ell = 3, 5$. For $p = 2$, there is only one $t \in \PS(1)$, namely $t = 0$, the trace of the trivial representation.  Monsky describes a different filtration of $K(\ell) = K(\ell)_0$ by Hecke-invariant subspaces, and proves that the Hecke algebras on the graded pieces are two copies of $A(1)$ plus a third ``new" Hecke algebra. The goal of this section is to compare the Monsky filtration to the standard filtration from \autoref{standardfilsec}, and to establish that the ``new" Monsky Hecke algebra coincides with $A(\ell)^\new$ defined here. The Monsky filtration exists more generally, so long as the level $\ell$ is congruent to $-1$ modulo $p$. As in the previous section, we will assume regularity conditions on $t$ (namely, $\cond(t, T_\ell)$), guaranteed in Monsky's $p = 2$ case by work of Nicolas and Serre \cite{NS2} (via \autoref{tellsurj}).

Fix a $t \in \PS(N)$, and let $\FF/\FF_p$ be an extension containing the image of $t$. Fix a prime $\ell$ congruent to $-1$ modulo $p$. Then we have the following filtration of $K(N\ell)_t$ by Hecke-invariant subspaces, due to Monsky \cite[remark p.~5]{monsky:level3fil}\footnote{The filtration that appers in Monsky's work is actually conjugated by $W_\ell$, namely: $$0 \subset W_\ell K(1) \subset \ker W_\ell \Tr W_\ell \subset K(\ell),$$ where the second-to-last term is the kernel of the map $W_\ell \Tr W_\ell: K(\ell) \to W_\ell K(1)$.}:
\begin{equation}\label{monskyfil}
0 \subset K(N)_t \subset (\ker \Tr)_t \subset K(N\ell)_t.
\end{equation}
Indeed, if $\ell + 1 = 0$ in $\FF_p$, then $\Tr\big(K(N)\big) = 0$, so that $(\ker \Tr)_t$ contains $K(N)_t$.

As in \autoref{standardfil}\eqref{tr}, if $\cond(t, T_\ell)$ holds, then the sequence
$$0 \to \ker \Tr \to K(N\ell)_t  \stackrel{\Tr}{\longrightarrow} K(N)_t \to 0$$
is exact. Therefore, the Hecke algebra on $K(N\ell)_t/(\ker \Tr)_t$ is isomorphic to $A(N)_t$.\footnote{As in \autoref{standardfilzerodiv}, condition $\zerodiv(t, T_\ell)$ suffices for the Hecke algebra conclusion.} Clearly, the Hecke algebra on $K(N)_t$ is $A(N)_t$ as well.

Let $K(N\ell)^{\monsky}_t$ be the Hecke module $(\ker \Tr)_t/K(N)_t$, and $A(N\ell)^{\monsky}_t$ be the Hecke algebra on $K(N\ell)^\monsky_t$.

\begin{myprop}\label{monskyfiltruth}
Suppose $\ell \equiv -1 \pmod{p}$ and $\cond(t, T_\ell)$ holds. The sequence $$0 \to \left.\ker T_\ell\right|_{K(N)_t} \to K(N\ell)_t^{\ellnew} \to K(N\ell)^\monsky_t \to 0$$ is exact, and induces an isomorphism of Hecke algebras $A(N\ell)^{\ellnew}_t \cong A(N\ell)^\monsky_t$.
\end{myprop}

\begin{proof}
Denote $\ker T_\ell|_{K(N)_t}$ by $(\ker T_\ell)_{N, t}$ below. We compare the exact sequences of the middle-graded piece of the Monsky filtration to the same from the standard filtration:

$$\xymatrix{
&0\ar@{-->}[d]
&0\ar[d]\\
0\ar@{-->}[r]
&(\ker T_\ell)_{N, t}\ar@{-->}[r]\ar@{-->}[d]
&K(N)_t\ar[d] \ar@{-->}[r]^{\ell T_\ell}
& K(N)_t \ar@{-->}[r]\ar@{==}[d]
&0\\
0  \ar[r]
&K(N\ell)^{\ellnew}_t  \ar[r]\ar@{-->}[d]
&(\ker \Tr)_t \ar[r]^{\Tr W_\ell}\ar[d]
&K(N)_t \ar[r]
&0\\
&K(N\ell)^{\ellnew}_t/(\ker T_\ell)_{N, t}\ar@{-->}[d]
&K(N\ell)^\monsky_t \ar[d] \\
&0
&0
}$$
Here the Monsky sequence is vertical with solid arrows and the standard sequence \eqref{standardmiddle} is horizontal with solid arrows. The inclusion $K(N)_t \into (\ker \Tr)_t$ from the Monsky sequence induces the upper horizontal exact sequence; note that the map $\Tr W_\ell$ restricted to $K(N)_t$ coincides with $\ell T_\ell$. Finally, the snake lemma on the resulting two horizontal short exact sequences gives us a natural isomorphism that we unpack as a short exact sequence below:
$$0 \to (\ker T_\ell)_{N, t} \to K(N\ell)^{\ellnew}_t \to K(N\ell)^\monsky_t \to 0.$$
The first map is natural inclusion; the second map is the composition $$K(N\ell)^{\ellnew} \into (\ker \Tr)_t \to (\ker \Tr)_t / K(N)_t = K(N\ell)^\monsky_t.$$
To see that the induced surjection on Hecke algebras $A(N\ell)^\ellnew_t \stackrel{\alpha}{\onto} A(N\ell)^\monsky_t$ is an isomorphism, we have to see that $A(N\ell)^{\ellnew}_t$ acts faithfully on $K(N\ell)^\monsky_t$. If $T = 0$ in $A(N\ell)^\monsky_t$, then $T$ sends $K(N\ell)^{\ellnew}_t$ to $(\ker T_\ell)_{N, t}$. Since $T$ commutes with $W_\ell$, it must also send $K(N\ell)^{\ellnew}_t = W_\ell K(N\ell)^{\ellnew}_t$ to $W_\ell(\ker T_\ell)_t$. Since $K(N)$ and $W_\ell K(N)$ are disjoint, $T$ must in fact annihilate all of $K(N\ell)^{\ellnew}_t$: that is, $T = 0$ in $A(N\ell)^{\ellnew}_t$.
\end{proof}

\subsection{Regularity conditions on the Hecke algebra $A(1)_t$}\label{tellsurjproof}

In this section we prove:

\begin{mylemma}\label{tellsurj}
If $A(1)_t \cong \FF\lb x, y \rb$, the action of any nonzero $T \in A(1)_t$ is surjective on $K(1)_t$.
\end{mylemma}

Note that $A(1)_t \cong \FF\lb x, y\rb$ if $t$ is \emph{unobstructed} in the sense of deformation theory. See \cite{NS2} for $p = 2$, \cite{BK} for $p \geq 5$, \cite{medved:ngt} for $p = 3$, and \cite{medved:thesis} for more discussion of $p = 2,3,5,7,13$.

\begin{proof}[Proof of {\autoref{tellsurj}}]
In level one, we have a perfect continuous duality between $A(1)$ and $K(1)$ as $A(1)$-modules under the pairing $A(1)_t \times K(1)_t \to \FF$ given by $\langle T, f \rangle := a_1(Tf)$. Therefore, we may choose a basis $\{m(a, b)\}_{a \geq 0, b \geq 0}$ of $K(1)$ dual to the ``Hilbert basis" $\{x^a y^b\}$: more precisely, one which satisfies $x \cdot m(0, b) = y \cdot m(a,0) = 0$ for all $a$, $b$, and $x \cdot m(a, b) = m(a-1, b)$ for $a \geq 1$ and $y\cdot m(a,b) = m(a, b-1)$ if $b \geq 1$.

We introduce a total order on pairs of nonnegative integers: we'll say that $(a, b) \prec (c, d)$ if $a + b < c + d$, or if $a + b = c + d$ and $b < d$. (In fact any total order will do.)
Suppose $$T = \sum_{a + b = k} c_{a, b} x^a y^b + O\big((x, y)^{k + 1}\big) \in \FF\lb x, y \rb$$
for some $k \geq 0$. 
Let $(a_0, b_0)$ be the $\prec$-minimal pair among all the pairs $(a, b)$ with $c_{a, b}$ nonzero; by scaling $T$ if necessary, we may assume that $c_{a_0, b_0} = 1$. For example, if $T_\ell = 5 x^2 y - y^3 + O\big((x, y)^4\big)$, then $(a_0, b_0) = (2, 1)$. We induct on $\prec$ to show that $m(a, b)$ is in the image of $T$ for any pair $(a, b)$. It's clear that $T \cdot m(a_0, b_0) = m(0, 0)$: base case. For the inductive step, suppose that the vector space $V_{a, b}= \big\langle m(c, d): (c,d) \prec (a, b) \big\rangle_\FF$ is in the image of $T$ already. Since $T\cdot m(a + a_0, b + b_0)$ is in $m(a, b) + V_{a, b}$, in fact $m(a, b)$ is in the image of $T$ as well.
\end{proof}

\begin{question}
Can one prove a similar statement for $A(N)_t$ if it is not a power series ring? At the very least, can one show that condition $\zerodiv(t, T)$ is satisfied?
\end{question}

\appendix

\section{The Atkin-Lehner automorphism mod $p$ geometrically\\ (Alexandru Ghitza)}\label{alexghitza}

Our aim is to describe a geometric construction of the modified Atkin-Lehner automorphism $W_\ell$ on the algebra of modular forms $M(\nl,\mathbb{F}_p)$.
This will be an intrinsic characteristic $p$ construction, stemming from an automorphism of the Igusa curve.

\subsection{Classical Atkin-Lehner via geometry}
Let's start by recalling the geometric construction of the Atkin-Lehner operator $w_\ell$, following Conrad~\cite{conrad}.

Let $\ell$ be a prime and $N$ a positive integer coprime to $\ell$.
The noncuspidal points on the modular curve $X_0(\nl)$ have the moduli interpretation
\begin{equation*}
\left(E;C_\ell,C_N\right)\quad\text{with $E$ an elliptic curve, $C_j$ cyclic subgroup of order $j$}.
\end{equation*}

We define an involution $w_\ell\colon Y_0(\nl)\to Y_0(\nl)$ by
\begin{equation*}
  (E;C_\ell,C_N)\mapsto
  (\phi(E);\phi(E[\ell]),\phi(C_\ell+C_N)),
\end{equation*}
where $\phi\colon E\to E/C_\ell$ is the quotient isogeny.

Conrad explains in what sense this involution can be extended to the cusps of $Y_0(\nl)$, and shows that over $\CC$, this construction yields the classical Atkin-Lehner involution on $M_k(\nl)$.
He also proves that, if $f(q)\in\ZZ[\frac{1}{\ell}]\lb q\rb$, then $(w_\ell f)(q)\in\ZZ[\frac{1}{\ell}]\lb q\rb$, from which we get the Atkin-Lehner involution $w_\ell$ on modular forms mod $p$ for any prime $p\neq\ell$.

As our setup is simpler (having the extra assumption that $p\nmid N$), we think of the classical mod $p$ Atkin-Lehner involution as coming directly from the map $w_\ell\colon\yonl\to\yonl$:
\begin{equation*}
  (E;C_\ell,C_N)\mapsto
  (\phi(E);\phi(E[\ell]),\phi(C_\ell+C_N)),
\end{equation*}
where $E$ is an elliptic curve in characteristic $p$ and $\phi\colon E\to E/C_\ell$ is the quotient isogeny.
More explicitly, if $f\in M_k(\nl,\mathbb{F}_p)$ and $\omega$ is a nonzero invariant differential on $E$, we have
\begin{equation*}
  (w_\ell f)(E;C_\ell,C_N,\omega)
  =f\left(\phi(E);\phi(E[\ell]),\phi(C_\ell+C_N),\widehat{\phi}^*\omega\right),
\end{equation*}
where $\widehat{\phi}\colon E/C_\ell\to E$ is the dual isogeny to $\phi\colon E\to E/C_\ell$.

\subsection{The Igusa curve $\ionl$}\label{igusacurve}
We summarize the features of Igusa curves that are essential to our construction.
We follow mainly Gross's exposition in~\cite[Section 5]{grosstame}, which develops the theory for $\Gamma_1(N)$-structure; this can be adapted to our $\Gamma_0(N)$ situation with minor changes, as summarized in~\cite[Section 10]{grosstame}.
A thorough study of Igusa curves appears in~\cite[Chapter 12]{katzmazur}, however without treatment of modular forms.
The $\Gamma_0(1)$ case is described briefly by Serre in~\cite[end of p.\:416-05]{SerreBBKmodp}; see also the discussion in \href{https://mathoverflow.net/questions/93059}{MathOverflow question 93059}.

Note that when $p=2$ we have $W_\ell=w_\ell$, the classical Atkin-Lehner automorphism.
We will henceforth assume that $p\geq 3$.

Consider a prime $p\neq\ell$ and coprime to $N$.
Given an elliptic curve $E$ in characteristic $p$, there are morphisms Frobenius $F\colon E\to E^{(p)}$ and Verschiebung $V\colon E^{(p)}\to E$ such that $V\circ F=[p]\colon E\to E$ and a canonical short exact sequence of group schemes
\begin{equation*}
  0\to \ker F\to E[p]\xrightarrow{F} \ker V\to 0.
\end{equation*}
An Igusa structure of level $p$ on $E$ is a choice of generator of (the Cartier divisor) $\ker V$.
This is equivalent to choosing a surjective morphism of group schemes $E[p]\to\ker V$, or (by Cartier duality) to choosing an embedding of group schemes $(\ker V)^*\hookrightarrow E[p]$.
We can be more precise by distinguishing the two cases:
\begin{itemize}
  \item If $E$ is ordinary, then $\ker V\cong\ZZ/p\ZZ$ and $(\ker V)^*\cong\mu_p$ so an Igusa structure is an embedding $\mu_p\hookrightarrow E[p]$.
  \item If $E$ is supersingular, then $\ker V\cong\alpha_p$ and $(\ker V)^*\cong\alpha_p$ so an Igusa structure is an embedding $\alpha_p\hookrightarrow E[p]$.
In fact, there is a unique such embedding (see~\cite[Example 3.14]{goren}).
\end{itemize}

If we restrict our attention to ordinary elliptic curves $E$, the moduli problem defined by the data
\begin{equation*}
  \left(E; C_\ell, C_N, \mu_p\xhookrightarrow{i_p} E[p]\right)
\end{equation*}
is representable (as we assume $p\geq 3$) by an affine curve $\iord$ whose coordinate ring we denote $S(\nl)$.
It has a natural smooth compactification $\ionl$ with a canonical map $\pi\colon \ionl\to \xonl$ that is totally ramified over the supersingular points.
It can be thought of as quotienting by the automorphism group $(\ZZ/p\ZZ)^\times/(\pm 1)$, which acts freely on $\iord$ via
\begin{equation*}
  \langle d\rangle_p
  \left(E;C_\ell,C_N,\mu_p\xhookrightarrow{i_p} E[p]\right)
  =
  \left(E;C_\ell,C_N,\mu_p\xhookrightarrow{i_p} E[p]\xrightarrow{[d]} E[p]\right).
\end{equation*}
This defines a grading on the algebra of functions
\begin{equation*}
  S(\nl)=\bigoplus_\alpha S_\alpha(\nl),
\end{equation*}
where $S_\alpha(\nl)$ consists of the functions on $\iord$ that satisfy $\langle d\rangle_p g=d^\alpha g$ for all $d\in (\ZZ/p\ZZ)^\times$.

The line bundle $\underline{\omega}^{\otimes 2}:=\Omega^1(\text{cusps})$ on $\xonl$ pulls back to a line bundle $\pi^*\underline{\omega}^{\otimes 2}$ on $\ionl$.
It is equipped with a canonical section $a^2$ with the following properties\footnote{We abuse notation by writing $a^2$ even though there is no $a$ itself for $\Gamma_0$-structures; so whenever we write $a^k$ we implicitly assume that $k$ is even and we set $a^k:=(a^2)^{(k/2)}$.}:
\begin{enumerate}
  \item $a^2$ is non-vanishing on $\iord$, and has simple zeros at the supersingular points;
  \item $a^{p-1}=\pi^* A$, where $A\in M_{p-1}(1,\mathbb{F}_p)$ is the Hasse invariant;
  \item $a^2$ has $q$-expansion $a^2(q)=1\in\mathbb{F}_p\lb q\rb$;
  \item $\langle d\rangle_p (a^2)=d^{-2}a^2$ for all $d\in(\ZZ/p\ZZ)^\times$.
\end{enumerate}
(This is the $\Gamma_0$-analogue of the $\Gamma_1$ result in~\cite[Proposition 5.2]{grosstame}, see also~\cite[Section 10]{grosstame}.)

We use the section $a^2$ to trivialize the line bundle $\pi^*\underline{\omega}^{\otimes 2}$.
This allows us to treat \emph{sections} of $\underline{\omega}^{\otimes k}$ on $\xonl$ as \emph{functions} on the ordinary locus $\iord$.
More precisely, the $q$-expansion map gives an isomorphism of graded $\mathbb{F}_p$-algebras
\begin{equation*}
  \Phi\colon S(\nl)\xrightarrow{\cong} M(\nl,\mathbb{F}_p)\subset\mathbb{F}_p\lb q\rb.
\end{equation*}
To see that the image of $\Phi$ is contained in $M(\nl,\mathbb{F}_p)$, let $g\in S_\alpha(\nl)$ and let $k\equiv\alpha\pmod{p-1}$ be such that $a^k g$ is regular on $\ionl$.
Since
\begin{equation*}
  \langle d\rangle_p (a^k g)=d^{-k}a^kd^\alpha g=d^{\alpha-k}(a^k g)=a^k g,
\end{equation*}
we see that $a^k g$ descends to a global section $f\in M_k(\nl,\mathbb{F}_p)$, and $g(q)=f(q)\in M(\nl,\mathbb{F}_p)^\alpha$.

For the inverse map: given $f(q)\in M(\nl,\mathbb{F}_p)^\alpha$, let $f\in M_k(\nl,\mathbb{F}_p)$ be any modular form with $q$-expansion $f(q)$, and let
\begin{equation*}
  g=\frac{\pi^* f}{a^k}
\end{equation*}
Then $g$ is a function on $\iord$ with $\langle d\rangle_p g=d^\alpha g$ and $g(q)=f(q)$.

\subsection{From maps on the Igusa curve to operators on modular forms mod $p$}
A morphism $\psi\colon\iord\to\iord$ on the ordinary locus of the Igusa curve determines a homomorphism of graded $\mathbb{F}_p$-algebras $\Psi\colon M(\nl,\mathbb{F}_p)\to M(\nl,\mathbb{F}_p)$ by setting
\begin{equation*}
  \Psi = \Phi\circ\psi^*\circ\Phi^{-1},
\end{equation*}
where, given $g\in S(\nl)$, $\psi^* g=g\circ\psi\in S(\nl)$.

For example, if we take $\psi=\langle\ell\rangle_p$ then $\Psi$ is the weight-separating automorphism $\wt_\ell$ defined in \autoref{weight_separating}.
To see this, let $f(q)\in M(\nl,\mathbb{F}_p)^\alpha$ and let $f$ be a modular form of weight $k\equiv\alpha\pmod{p-1}$ with $q$-expansion $f(q)$; we have
\begin{equation*}
  \Psi(f(q)) = \Phi\left(\langle\ell\rangle_p\left(\frac{\pi^* f}{a^k}\right)\right)
  =\Phi\left(\frac{\pi^* f}{(\ell^{-2} a^2)^{k/2}}\right)
  =\ell^k \Phi\left(\frac{\pi^* f}{a^k}\right)
  =\ell^\alpha f(q)
  =\wt_\ell f(q).
\end{equation*}

In order to recover the modified Atkin-Lehner automorphism $W_\ell$ defined in \autoref{renormalizeal}, we start with the map $\widetilde{w}_\ell\colon\iord\to\iord$ given by
\begin{equation*}
  \widetilde{w}_\ell \left(E; C_\ell, C_N, i_p\right)
  =
  \left(\phi(E); \phi(E[\ell]), \phi(C_\ell+C_N), \phi\circ i_p\right),
\end{equation*}
where $\phi\colon E\to E/C_\ell$ is the quotient isogeny.
Since
\begin{equation*}
  \widetilde{w}^2_\ell \left(E; C_\ell, C_N, i_p\right)
  =\left(E;C_\ell,C_N,\widehat{\phi}\circ\phi\circ i_p\right),
\end{equation*}
we conclude that $\widetilde{w}^2_\ell = \langle\ell\rangle_p$.

We denote the corresponding algebra homomorphism $\Psi$ resulting from $\psi=\widetilde{w}_\ell$ by $\widetilde{W}_\ell$, so that
\begin{equation*}
  \widetilde{W}_\ell=\Phi\circ\widetilde{w}_\ell^*\circ\Phi^{-1}.
\end{equation*}

\begin{mylemma} \label{wltilde_pi}
  For any modular form $f$ of level $\Gamma_0(\nl)$ we have
  $\widetilde{w}_\ell^*(\pi^* f)=\pi^*(w_\ell f)$.
\end{mylemma}
\begin{proof}
  This is a simple calculation on the moduli:
\begin{align*}
  \widetilde{w}_\ell^*(\pi^* f)(E;C_\ell,C_N,i_p,\omega)
  &=(\pi^* f)(\phi(E);\phi(E[\ell]),\phi(C_\ell+C_N),\phi\circ i_p,\widehat{\phi}^*\omega)\\
  &=f(\phi(E);\phi(E[\ell]),\phi(C_\ell+C_N),\widehat{\phi}^*\omega)\\
  &=(w_\ell f)(E;C_\ell,C_N,\omega)\\
  &=\pi^* (w_\ell f) (E;C_\ell,C_N,i_p,\omega)
\end{align*}
\end{proof}

\begin{mylemma} \label{wla2}
  $\widetilde{w}_\ell^*(a^2)=\ell^{-1} a^2$ as elements of $H^0(\ionl,\pi^*\underline{\omega}^{\otimes 2})$.
\end{mylemma}
\begin{proof}
  We temporarily pass to $\Gamma_1(\nl)$-structures and work with the corresponding Igusa covering $\pi\colon I_1(\nl)\to X_1(\nl)_{\mathbb{F}_p}$ of degree $p-1$.
  This setting has the advantage that we obtain a $(p-1)$-st root $a_1$ of the Hasse invariant $A$, as a canonical section of the line bundle $\pi^*\underline{\omega}$, as detailed in~\cite[Proposition 5.2]{grosstame}.
  Given a choice of $\ell$-th root of unity $\zeta$, \cite[Section 6]{grosstame} defines an automorphism $w_\zeta$ of $X_1(\nl)$ by giving a modular recipe
\begin{equation*}
  w_\zeta(E;\beta_\ell,\alpha_N)
  =(\phi(E);\phi(\beta_\ell),\phi(\alpha_N)),
\end{equation*}
where $\phi(\alpha_N)=\phi\circ\alpha_N\colon\mu_N\into\phi(E)[N]$.
The definition of $\phi(\beta_\ell)$ is more intricate, and involves the choice of $\ell$-th root of unity $\zeta$.
If $e\colon E[\ell]\times E[\ell]\to\mu_\ell$ denote the Weil pairing, there is a unique $P_\beta\in E[\ell]/\beta_\ell(\mu_\ell)=\phi(E[\ell])$ such that
\begin{equation*}
  e(\beta_\ell(z),P_\beta)=z\qquad\text{for all }z\in\mu_\ell.
\end{equation*}
Let $\phi(\beta_\ell)\colon\mu_\ell\to\phi(E)[\ell]$ be defined by $\phi(\beta_\ell)(\zeta)=\phi(P_\beta)$.
Note that $\phi(\beta_\ell)(\mu_\ell)=\phi(E[\ell])$.

We can adapt this into an automorphism $\widetilde{w}_\zeta$ of $I_1(\nl)^{\text{ord}}$ by setting
\begin{equation*}
  \widetilde{w}_\zeta(E;\beta_\ell,\alpha_N,i_p)
  =(\phi(E);\phi(\beta_\ell),\phi(\alpha_N),\phi(i_p))
\end{equation*}
where $\phi(i_p)=\phi\circ i_p\colon \mu_p\hookrightarrow \phi(E)[p]$.

We illustrate the various spaces and maps in the following cube diagram whose commutativity is readily checked via calculations similar to that in \autoref{wltilde_pi}, using the moduli interpretation of the covering maps $\eta\colon I_1(\nl)\to I_0(\nl)$:
\begin{equation*}
  \eta(E;\beta_\ell,\alpha_N,i_p)=(E;\beta_\ell(\mu_\ell),\alpha_N(\mu_N),i_p)
\end{equation*}
and similarly for $\eta\colon X_1(\nl)\to X_0(\nl)$.
\begin{equation*}
\begin{tikzcd}[row sep=2.5em]
  &
    {I_1(\nl)}
    \arrow[rr,"\textcolor{red}{\widetilde{w}_\zeta}"]
    \arrow[dl,swap,"\textcolor{red}{\pi}"]
    \arrow[dd,"\textcolor{red}{\eta}" near end]
  &&
    {I_1(\nl)}
    \arrow[dd,"\textcolor{red}{\eta}"]
    \arrow[dl,"\textcolor{red}{\pi}"]
\\
    {X_1(\nl)}
    \arrow[rr,crossing over,"\textcolor{red}{w_\zeta}" near end]
    \arrow[dd,swap,"\textcolor{red}{\eta}"]
  &&
    {X_1(\nl)}
\\
  &
    {I_0(\nl)}
    \arrow[rr,swap,"\textcolor{red}{\widetilde{w}_\ell}" near start]
    \arrow[dl,swap,"\textcolor{red}{\pi}"]
  &&
    {I_0(\nl)}
    \arrow[dl,"\textcolor{red}{\pi}"]
\\
    {X_0(\nl)}
    \arrow[rr,swap,"\textcolor{red}{w_\ell}"]
  &&
    {X_0(\nl)}
    \arrow[uu,<-,crossing over,swap,"\textcolor{red}{\eta}" near end]
\end{tikzcd}
\end{equation*}
Our immediate interest is in the back face of the cube, so we spell out its commutativity:
\begin{align*}
  \eta\circ\widetilde{w}_\zeta (E;\beta_\ell,\alpha_N,i_p)
  &= \big(\phi(E);\phi(\beta_\ell)(\mu_\ell),\phi(\alpha_N)(\mu_N),\phi(i_p)\big)\\
  \widetilde{w}_\ell\circ\eta (E;\beta_\ell,\alpha_N,i_p)
  &= \big(\phi(E);\phi(E[\ell]),\phi\big(\beta_\ell(\mu_\ell)+\alpha_N(\mu_N)\big),\phi(i_p)\big),
\end{align*}
where $\phi\colon E\to E/\beta_\ell(\mu_\ell)$ is the quotient isogeny.
We observed above that $\phi(\beta_\ell)(\mu_\ell)=\phi(E[\ell])$ from the definition of $\phi(\beta_\ell)$; it remains to see that $\phi(\alpha_N)(\mu_N)=\phi\big(\beta_\ell(\mu_\ell)+\alpha_N(\mu_N)\big)$, which simply follows from $\beta_\ell(\mu_\ell)$ being killed by $\phi$.

So $\widetilde{w}_\ell\circ\eta=\eta\circ\widetilde{w}_\zeta$, which combined with the surjectivity of $\eta$ and the following Lemma, yields the claim $\widetilde{w}_\ell^*(a^2)=\ell^{-1}a^2$.
\end{proof}

\begin{mylemma} \label{wza2}
  $\widetilde{w}_\zeta^*(a_1^2)=\ell^{-1} a_1^2$ as elements of $H^0(I_1(\nl),\pi^*\underline{\omega}^{\otimes 2})$.
\end{mylemma}
\begin{proof}
Let $g=\widetilde{w}_\zeta^*(a_1)$, then
\begin{equation*}
  g^{p-1}=\widetilde{w}_\zeta^*(a_1^{p-1})=\widetilde{w}_\zeta^*(\pi^* A)=\pi^*(w_\zeta A).
\end{equation*}
Passing to $q$-expansions and recalling that $(w_\zeta A)(q)=\ell^{(p-1)/2}A(q^\ell)=\ell^{(p-1)/2}$, we get that $g(q)^{p-1}=\ell^{(p-1)/2}$.
The crucial point is that the $q$-expansion $g(q)^{p-1}$ is a constant, which implies that $g(q)$ itself is a constant, which we will call $\gamma$ for short.
The space $H^0(I_1(\nl),\pi^*\underline{\omega})$ has another element whose $q$-expansion is $\gamma$, namely $\gamma a_1$.
So by the $q$-expansion principle, we conclude that $\widetilde{w}_\zeta^*(a_1)=\gamma a_1$.
Upon iterating, we get
\begin{equation*}
  \ell^{-1} a_1=\langle\ell\rangle_p(a_1)=(\widetilde{w}_\zeta^*)^2(a_1)=\widetilde{w}_\zeta^*(\gamma a_1)=\gamma^2 a_1,
\end{equation*}
so that $\gamma^2=\ell^{-1}$.
We conclude that
\begin{equation*}
  \widetilde{w}_\zeta^*(a_1^2)=(\gamma a_1)^2=\ell^{-1} a_1^2.
\end{equation*}
\end{proof}

\begin{remark}
  The reader is perhaps wondering why we had to involve $\Gamma_1$-structures.
  It is indeed possible to apply the argument in \autoref{wza2} directly to the trivializing section $a^2$ on $I_0(\nl)$, but that only allows us to conclude that $\widetilde{w}_\ell^*(a^2)=\pm\ell^{-1}a^2$, and we are unable to rule out the possible negative sign when $p\equiv 1\pmod{4}$.
  The $\Gamma_1$ setting provides us with a square root of $a_1^2$, which strenghens the argument enough to rule out the unwanted $-1$.
  It is possible that working with the moduli \emph{stack} $\mathcal{X}_0(\nl)$ instead of the \emph{coarse moduli space} $X_0(\nl)$ could also provide the needed flexibility, without the artifice of changing level structures.
\end{remark}

\begin{myprop}
  If $f$ is a modular form of weight $k$ and $q$-expansion $f(q)$, we have
  \begin{equation*}
    \widetilde{W}_\ell f(q)=W_\ell f(q).
  \end{equation*}
\end{myprop}
\begin{proof}
  This is just a matter of combining \autoref{wltilde_pi} and \autoref{wla2}:
\begin{equation*}
  \widetilde{W}_\ell f(q)
  =\Phi\left(\widetilde{w}_\ell^*\left(\frac{\pi^* f}{a^k}\right)\right)
  =\Phi\left(\frac{\widetilde{w}_\ell^*(\pi^* f)}{\widetilde{w}_\ell^*(a^k)}\right)
  =\Phi\left(\frac{\pi^*(w_\ell f)}{\ell^{-k/2}a^k}\right)
  =\ell^{k/2} w_\ell f(q)
  =W_\ell f(q).
\end{equation*}
\end{proof}


\end{document}